\documentclass[a4paper,twoside,10pt]{amsart}
\usepackage{amsmath,amsfonts,amssymb,amsthm}
\newtheorem{theorem}{Theorem}[section]
\newtheorem{lemma}{Lemma}[section]
\newtheorem{proposition}{Proposition}[section]
\usepackage{color,graphicx}
\usepackage[a4paper]{geometry}
\numberwithin{equation}{section}

\author[G. Nemes]{Gerg\H{o} Nemes}
\address{School of Mathematics, The University of Edinburgh, James Clerk Maxwell Building, The King's Buildings, Peter Guthrie Tait Road, Edinburgh EH9 3FD, UK}
\email{Gergo.Nemes@ed.ac.uk}


\begin{document}

\title[The large-argument asymptotics of the Lommel and allied functions]{Error bounds for the large-argument asymptotic expansions of the Lommel and allied functions}

\begin{abstract} In this paper, we reconsider the large-$z$ asymptotic expansion of the Lommel function $S_{\mu,\nu}(z)$ and its derivative. New representations for the remainder terms of the asymptotic expansions are found and used to obtain sharp and realistic error bounds. We also give re-expansions for these remainder terms and provide their error estimates. Applications to the asymptotic expansions of the Anger--Weber-type functions, the Scorer functions, the Struve functions and their derivatives are provided. The sharpness of our error bounds is discussed in detail, and numerical examples are given.
\end{abstract}
\maketitle

\section{Introduction and main results}\label{section1}

In the present paper, we reconsider the large-$z$ asymptotic expansion of the Lommel function $S_{\mu,\nu}(z)$, which is the unique solution of the inhomogeneous Bessel differential equation
\[
\frac{d^2 w(z)}{dz^2} + \frac{1}{z}\frac{dw(z)}{dz} + \left( 1 - \frac{\nu ^2}{z^2} \right)w(z) = z^{\mu  - 1},
\]
having the property that $S_{\mu,\nu}(z)\sim z^{\mu-1}$ as $z\to \infty$ in the sector $|\arg z|\leq \pi-\delta$ with $\delta >0$ being fixed. The orders $\mu$ and $\nu$ of this function can take arbitrary complex values. As a function of $z$, with fixed $\mu$ and $\nu$, $S_{\mu,\nu}(z)$ is analytic in the domain $|\arg z|< \pi$ and (in general) possesses a branch-point singularity at the origin. The Lommel function plays an important role in the evaluation of certain integrals involving Bessel functions (see, for example, \cite[Ch. III]{Luke}).

It was shown by Watson \cite[\S 10.75]{Watson} in 1922 that, as $z\to \infty$ in the sector $|\arg z|\leq \pi-\delta$ with $\delta >0$ being fixed, the Lommel function has the asymptotic expansion
\begin{equation}\label{eq1}
S_{\mu ,\nu }(z)  \sim z^{\mu  - 1} \sum\limits_{n = 0}^\infty  ( - 1)^n \frac{a_n ( - \mu ,\nu )}{z^{2n}},
\end{equation}
where
\[
a_n (\mu ,\nu ) = \prod\limits_{k = 1}^n ((\mu  + 2k - 1)^2  - \nu ^2 )  = 2^{2n} \left( \frac{\mu  + \nu  + 1}{2} \right)_n \left( \frac{\mu  - \nu  + 1}{2} \right)_n ,
\]
with the Pochhammer symbol $(w)_n = \Gamma(w + n)/\Gamma(w)$. If either of $\mu\pm\nu$ equals a positive odd integer, then the right-hand side of \eqref{eq1} terminates and represents $S_{\mu,\nu}(z)$ exactly.

Since $S_{\mu,\nu}(z)$ is analytic in the domain $| \arg z| < \pi$, the asymptotic expansion \eqref{eq1} can be differentiated term-wise and yields
\begin{equation}\label{eq2}
S'_{\mu ,\nu } (z) \sim z^{\mu  - 2} \sum\limits_{n = 0}^\infty  ( - 1)^n \frac{b_n ( - \mu ,\nu )}{z^{2n}},
\end{equation}
as $z\to \infty$ in the sector $|\arg z|\leq \pi-\delta$ with any fixed $\delta >0$. The coefficients in this expansion are given by
\[
b_n (\mu ,\nu ) = - a_n (\mu ,\nu )(\mu  + 2n + 1).
\]

The main aim of the present paper is to derive new representations and bounds for the remainders of the asymptotic expansions \eqref{eq1} and \eqref{eq2}. Thus, for $| \arg z| < \pi$, complex $\mu$ and $\nu$, and any non-negative integer $N$, we define the $N$th remainder terms $R_N^{(S)}(z,\mu,\nu)$ and $R_N^{(S')}(z,\mu,\nu)$ of the asymptotic expansions \eqref{eq1} and \eqref{eq2} via the equalities
\begin{equation}\label{eq22}
S_{\mu ,\nu }(z) = z^{\mu  - 1} \left(\sum\limits_{n = 0}^{N-1} ( - 1)^n \frac{a_n ( - \mu ,\nu )}{z^{2n}} + R_N^{(S)}(z,\mu,\nu) \right)
\end{equation}
and
\begin{equation}\label{eq76}
S'_{\mu ,\nu } (z) = z^{\mu  - 2} \left(\sum\limits_{n = 0}^{N-1} ( - 1)^n \frac{b_n ( - \mu ,\nu )}{z^{2n}} + R_N^{(S')}(z,\mu,\nu) \right).
\end{equation}
Throughout this paper, if not stated otherwise, empty sums are taken to be zero. The derivations of the estimates for $R_N^{(S)}(z,\mu,\nu)$ and $R_N^{(S')}(z,\mu,\nu)$ are based on new representations of these remainder terms.

The remainder $R_N^{(S)}(z,\mu,\nu)$ has already been investigated in an earlier paper of the author \cite{Nemes1}. The results of the paper \cite{Nemes1} are special cases or consequences of the results of the present work.

Before stating the main results in detail, we introduce some notation. We denote
\[
\Pi_p(w)=\frac{w^p}{2}\left( e^{\frac{\pi }{2}ip} e^{iw} \Gamma(1 - p,we^{\frac{\pi }{2}i} ) + e^{ - \frac{\pi }{2}ip} e^{ - iw} \Gamma ( 1 - p,we^{ - \frac{\pi }{2}i} ) \right),
\]
where $\Gamma(1 - p,w)$ is the incomplete gamma function. The function $\Pi_p(w)$ was originally introduced by Dingle \cite[pp. 407]{Dingle} and, following his convention, we refer to it as a basic terminant (but note that Dingle's notation slightly differs from ours, e.g., $\Pi_{p-1}(w)$ is used for our $\Pi_p(w)$). The basic terminant is a multivalued function of its argument $w$ and, when the argument is fixed, is an entire function of its order $p$. In this paper, we will use the integral representation \cite{Nemes2}
\begin{equation}\label{eq01}
\Pi_p (w) = \frac{1}{\Gamma (p)}\int_0^{ + \infty } \frac{t^{p - 1} e^{ - t}}{1 + (t/w)^2 }dt,
\end{equation}
which is valid when $|\arg w|<\frac{\pi}{2}$ and $\Re(p)>0$.

We shall also use the concept of the regularized hypergeometric function ${\bf F}\left( {a,b;c;w} \right)$ which is defined by the power series expansion
\[
{\bf F} ( a,b;c;w ) = \frac{1}{\Gamma ( a )\Gamma ( b )}\sum\limits_{n = 0}^\infty  \frac{\Gamma ( a + n )\Gamma ( b + n )}{\Gamma (c + n )\Gamma (n + 1)}w^n 
\]
for $\left| w \right| < 1$ and by analytic continuation elsewhere \cite[\S 15.2]{NIST}. The parameters $a$, $b$ and $c$ of this function can take arbitrary complex values.

We are now in a position to formulate our main results. In Theorem \ref{thm1}, we give a new integral representation for the remainder term $R_N^{(S)}(z,\mu,\nu)$ involving a free complex parameter $\lambda$, $\Re \left( \lambda  \right) > 0$.

\begin{theorem}\label{thm1} Let $N$ be a non-negative integer and let $\mu$, $\nu$ and $\lambda$ be arbitrary complex numbers such that $\Re(\mu)+|\Re(\nu)| < 2N +1$, $\Re (\mu ) < 2N + \Re (\lambda ) + \frac{1}{2}$ and $\Re (\lambda) > 0$. Then
\begin{gather}\label{eq6}
\begin{split}
R_N^{(S)} (z,\mu ,\nu ) = \; & ( - 1)^N \frac{2^{\mu  + \frac{1}{2}} \pi ^{\frac{1}{2}} \Gamma \left( 2N - \mu  + \lambda  + \frac{1}{2} \right)}{\Gamma \left( \frac{ - \mu  + \nu  + 1}{2} \right)\Gamma \left( \frac{ - \mu  - \nu  + 1}{2} \right)}\frac{1}{z^{2N}} \\ & \times \int_0^{ + \infty } \frac{t^{\lambda  - 1} }{(1 + t)^{2N - \mu  + \lambda  + \frac{1}{2}} }{\bf F}\left( \nu  + \frac{1}{2}, - \nu  + \frac{1}{2};\lambda ; - \frac{t}{2} \right)\Pi _{2N - \mu  + \lambda  + \frac{1}{2}} (z(1 + t))dt,
\end{split}
\end{gather}
provided $|\arg z|<\pi$.
\end{theorem}

In the special case that $\lambda  =  - \nu  + \frac{1}{2}$, formula \eqref{eq6} was also derived by Dingle (\cite[eq. (9)]{Dingle0}, \cite[eq. (45), p. 442]{Dingle}) using formal, non-rigorous methods.

If $\lambda=0$, a slight modification of \eqref{eq6} holds.

\begin{theorem}\label{thm2} Let $N$ be a non-negative integer and let $\mu$ and $\nu$ be arbitrary complex numbers such that $\Re(\mu)+|\Re(\nu)| < 2N +1$ and $\Re (\mu ) < 2N + \frac{1}{2}$. Then
\begin{gather}\label{eq33}
\begin{split}
& R_N^{(S)} (z,\mu ,\nu ) = ( - 1)^N \frac{2^{\mu  + \frac{1}{2}} \pi ^{\frac{1}{2}} \Gamma \left( 2N - \mu  + \frac{1}{2} \right)}{\Gamma \left( \frac{ - \mu  + \nu  + 1}{2} \right)\Gamma \left( \frac{ - \mu  - \nu  + 1}{2} \right)}\frac{1}{z^{2N} } \\ & \times \left( \Pi _{2N - \mu  + \frac{1}{2}} (z) + \int_0^{ + \infty } \frac{t^{ - 1} }{(1 + t)^{2N - \mu  + \frac{1}{2}} }{\bf F}\left( {\nu  + \frac{1}{2}, - \nu  + \frac{1}{2};0; - \frac{t}{2}} \right)\Pi _{2N - \mu  + \frac{1}{2}} (z(1 + t))dt  \right),
\end{split}
\end{gather}
provided $|\arg z|<\pi$.
\end{theorem}

Analogous expressions for the remainder term $R_N^{(S')}(z,\mu,\nu)$ can be written down by applying Theorems \ref{thm1} and \ref{thm2} together with the functional equation
\begin{equation}\label{eq20}
2R_N^{(S')} (z,\mu ,\nu ) = (\mu  + \nu  - 1)R_N^{(S)} (z,\mu  - 1,\nu  - 1) + (\mu  - \nu  - 1)R_N^{(S)} (z,\mu  - 1,\nu  + 1).
\end{equation}
This functional equation follows directly from the connection formula $2S'_{\mu ,\nu } (z) = (\mu  + \nu  - 1)S_{\mu  - 1,\nu  - 1} (z) + (\mu  - \nu  - 1)S_{\mu  - 1,\nu  + 1} (z)$ (see, e.g., \cite[eq. (10), p. 348]{Watson}).

It was shown by the present author \cite{Nemes1} that for any non-negative integer $N$, the remainder term $R_N^{(S)} (z,\mu ,\nu )$ can be expressed in terms of the modified Bessel function as
\begin{equation}\label{eq3}
R_N^{(S)} (z,\mu ,\nu ) = ( - 1)^N \frac{2^{\mu  + 1} }{\Gamma \left( \frac{ - \mu  + \nu  + 1}{2} \right)\Gamma \left( \frac{ - \mu  - \nu  + 1}{2} \right)}\frac{1}{z^{2N}}\int_0^{ + \infty }\frac{t^{2N - \mu } K_\nu  (t)}{1 + (t/z)^2 }dt ,
\end{equation}
provided that $|\arg z|<\frac{\pi}{2}$ and $\Re(\mu)+|\Re(\nu)| < 2N +1$. This representation of $R_N^{(S)} (z,\mu ,\nu )$ is the central tool of the paper \cite{Nemes1} in the asymptotic analysis of the Lommel function $S_{\mu,\nu}(z)$, in particular, in the derivation of bounds for $R_N^{(S)} (z,\mu ,\nu )$. In the following theorem, we present a result for the remainder term $R_N^{(S')} (z,\mu ,\nu )$ analogous to \eqref{eq3}. In our analysis, the purpose of these representations is not to find bounds for the remainders themselves but to prove error bounds for the re-expansions of these remainders.

\begin{theorem}\label{thm3} Let $N$ be a non-negative integer and let $\mu$ and $\nu$ be arbitrary complex numbers such that $\Re(\mu)+|\Re(\nu)| < 2N +1$. Then
\begin{equation}\label{eq21}
R_N^{(S')} (z,\mu ,\nu ) = ( - 1)^N \frac{2^{\mu  + 1} }{\Gamma \left( \frac{ - \mu  + \nu  + 1}{2} \right)\Gamma \left( \frac{ - \mu  - \nu  + 1}{2} \right)}\frac{1}{z^{2N}}\int_0^{ + \infty } \frac{t^{2N - \mu  + 1} K'_\nu  (t)}{1 + (t/z)^2 }dt,
\end{equation}
provided $|\arg z|<\frac{\pi}{2}$.
\end{theorem}

The subsequent theorem provides bounds for the remainders $R_N^{(S)} (z,\mu ,\nu )$ and $R_N^{(S')} (z,\mu ,\nu )$ when $\mu$ and $\nu$ are real. These error bounds may be further simplified by employing the various estimates for $\mathop {\sup }\nolimits_{r \geq 1} | \Pi _{p} (zr) |$ given in Appendix \ref{appendixa}. Note that those estimates in Appendix \ref{appendixa} that depend on the (positive) order $p$ are monotonically increasing functions of $p$. Therefore, in the following theorem, we minimize, as much as our methods allow, the order of the basic terminants with respect to the following natural requirement: the form of each bound is directly related to the first omitted term of the corresponding asymptotic expansion. We remark that the bound \eqref{eq27} remains true even if the order of the basic terminant is taken to be any positive quantity at least $2N - \mu  + \frac{3}{2} +\max \left( 0,\frac{1}{2} - \left| \nu  \right| \right)$.

\begin{theorem}\label{thm4}
Let $N$ be a non-negative integer and let $\mu$, $\nu$ and $\lambda$ be arbitrary real numbers such that $\mu+|\nu| < 2N +1$, $\mu < 2N +\lambda+ \frac{1}{2}$ and $\lambda  \ge \max \left( 0,\frac{1}{2} - \left| \nu  \right| \right)$. Then
\begin{equation}\label{eq60}
\big| R_N^{(S)} (z,\mu ,\nu ) \big| \le \frac{\left| a_N ( - \mu ,\nu ) \right|}{\left| z \right|^{2N} }\mathop {\sup }\limits_{r \ge 1} \big| \Pi _{2N - \mu  + \lambda + \frac{1}{2}} (zr) \big|,
\end{equation}
provided $|\arg z|<\pi$. Similarly, let $N$ be a non-negative integer and let $\mu$ and $\nu$ be arbitrary real numbers such that $\mu+|\nu| < 2N +1$. Then
\begin{equation}\label{eq27}
\big| R_N^{(S')} (z,\mu ,\nu ) \big| \le \frac{\left| b_N ( - \mu ,\nu ) \right|}{\left| z \right|^{2N} }\mathop {\sup }\limits_{r \ge 1} \big| \Pi _{2N - \mu  + \frac{3}{2}} (zr) \big|,
\end{equation}
provided $|\arg z|<\pi$.
\end{theorem}

It is known that in the special case when $z$ is positive and both $\mu$ and $\nu$ are real, $\mu +|\nu|<2N+1$, we have
\begin{equation}\label{eq79}
R_N^{(S)} (z,\mu ,\nu ) = \frac{a_N ( - \mu ,\nu ) }{z^{2N}} \theta_N^{(S)}(z,\mu,\nu).
\end{equation}
Here $0<\theta_N^{(S)}(z,\mu,\nu)<1$ is an appropriate number that depends on $z$, $\mu$, $\nu$ and $N$ (see \cite{Nemes1}). Said differently, the remainder term
$R_N^{(S)} (z,\mu ,\nu )$ does not exceed the corresponding first neglected term in absolute value and has the same sign provided that $z > 0$ and $\mu +|\nu|<2N+1$. In the theorem below, we present the analogous result for the remainder term $R_N^{(S')} (z,\mu ,\nu )$.

\begin{theorem}\label{thm5}
Let $N$ be a non-negative integer, $z$ be a positive real number, and let $\mu$ and $\nu$ be arbitrary real numbers such that $\mu +|\nu|<2N+1$. Then
\begin{equation}\label{eq41}
R_N^{(S')} (z,\mu ,\nu ) = \frac{b_N ( - \mu ,\nu ) }{z^{2N}} \theta_N^{(S')}(z,\mu,\nu),
\end{equation}
where $0<\theta_N^{(S')}(z,\mu,\nu)<1$ is a suitable number that depends on $z$, $\mu$, $\nu$ and $N$.
\end{theorem}

We now consider the case when $\mu$ and $\nu$ are complex. In the following theorem, we have chosen $2N - \mu  + 1$, respectively $2N - \mu  + 2$, as the order of the basic terminants because this is the value that allows us to express the bounds in a form closely related to the first omitted term in the corresponding asymptotic expansion. The precise relation to the first omitted term is examined in Section \ref{section6}.

\begin{theorem}\label{thm6}
Let $N$ be a non-negative integer and let $\mu$ and $\nu$ be arbitrary complex numbers such that $\Re(\mu)+|\Re(\nu)| < 2N +1$. Then
\begin{gather}\label{eq43}
\begin{split}
\big| R_N^{(S)} (z,\mu ,\nu ) \big| \le \; & \left| \frac{\Gamma \big( \frac{ - \Re (\mu ) + \Re (\nu ) + 1}{2} \big)\Gamma \big( \frac{ - \Re (\mu ) - \Re (\nu ) + 1}{2} \big)}{\Gamma \left( \frac{ - \mu  + \nu  + 1}{2} \right)\Gamma \left( \frac{ - \mu  - \nu  + 1}{2} \right)}\frac{\Gamma ( 2N - \mu  + 1 )}{\Gamma ( 2N - \Re (\mu ) + 1 )} \right|  \\ & \times \frac{\left| a_N ( - \Re (\mu ),\Re (\nu )) \right|}{\left| z \right|^{2N} }\mathop {\sup }\limits_{r \ge 1} \left| {\Pi _{2N - \mu  + 1} (zr)} \right|
\end{split}
\end{gather}
and
\begin{gather}\label{eq44}
\begin{split}
\big| R_N^{(S')} (z,\mu ,\nu ) \big| \le \; & \left| \frac{\Gamma \big( \frac{ - \Re (\mu ) + \Re (\nu ) + 1}{2} \big)\Gamma \big( \frac{ - \Re (\mu ) - \Re (\nu ) + 1}{2} \big)}{\Gamma \left( \frac{ - \mu  + \nu  + 1}{2} \right)\Gamma \left( \frac{ - \mu  - \nu  + 1}{2} \right)}\frac{\Gamma ( 2N - \mu  + 2 )}{\Gamma ( 2N - \Re (\mu ) + 2 )} \right|  \\ & \times  \frac{\left| b_N ( - \Re (\mu ),\Re (\nu )) \right|}{\left| z \right|^{2N} }\mathop {\sup }\limits_{r \ge 1} \left| {\Pi _{2N - \mu  + 2} (zr)} \right|,
\end{split}
\end{gather}
provided $|\arg z|<\pi$. If $\Re(\mu)\pm\Re(\nu) \neq \mu \pm \nu$ and at least one of them is a positive odd integer, then the limiting value has to be taken in these bounds.
\end{theorem}

A further simplification of these bounds is possible by employing the estimate
\[
\left| \frac{\Gamma \big( \frac{ - \Re (\mu ) + \Re (\nu ) + 1}{2} \big)\Gamma \big( \frac{ - \Re (\mu ) - \Re (\nu ) + 1}{2} \big)}{\Gamma \left( \frac{ - \mu  + \nu  + 1}{2} \right)\Gamma \left( \frac{ - \mu  - \nu  + 1}{2} \right)} \right| \le \left| \frac{\cos (\pi \mu ) + \cos (\pi \nu )}{\cos (\pi \Re (\mu )) + \cos (\pi \Re (\nu ))} \right|,
\]
which follows from the reflection formula for the gamma function and the inequality $|\Gamma(w)|\leq \Gamma(\Re(w))$ \cite[eq. 5.6.6]{NIST}.

We would like to emphasize that the requirement $\Re(\mu)+|\Re(\nu)| < 2N +1$ in the above theorems is not a serious restriction. Indeed, the index of the numerically least term of the asymptotic expansion \eqref{eq1}, for example, is $n \approx \frac{1}{2}|z|$. Therefore, it is reasonable to choose the optimal $N \approx \frac{1}{2}|z|$, whereas the condition $\mu \pm \nu = o(|z|)$ has to be fulfilled in order to obtain proper approximations from \eqref{eq22}.

A detailed discussion on the sharpness of our error bounds and some numerical examples are given in Section \ref{section6}.

In the following, we consider re-expansions for the remainders of the asymptotic expansions of the Lommel function and its derivative. A re-expansion for the remainder term of the asymptotic expansion of the function $S_{\mu,\nu}(z)$, in order to improve its numerical efficacy, was first derived, using formal methods, by Dingle \cite{Dingle0} in 1959. His work was placed on rigorous mathematical foundations by the present author \cite{Nemes1} in 2015, who derived an exponentially improved asymptotic expansion for the Lommel function $S_{\mu,\nu}(z)$ valid when $|\arg z|\leq \frac{\pi}{2}$. Here, we shall reconsider this result and obtain explicit bounds for the error term of the expansion. An analogous result for the remainder $R_N^{(S')} (z,\mu ,\nu )$ is also provided.

In order to state our results, we require the truncated versions of the well-known large-$z$ asymptotic expansions of the modified Bessel function $K_\nu(z)$ and its derivative $K'_\nu(z)$. Thus, for $|\arg z|<\pi$, complex $\nu$, and any non-negative integer $M$, we define the $M$th remainder terms $R_M^{(K)} (z,\nu )$ and $R_M^{(K')} (z,\nu )$ via
\begin{equation}\label{eq54}
K_\nu(z) = \left( \frac{\pi}{2z} \right)^{\frac{1}{2}} e^{ - z} \left( \sum\limits_{m = 0}^{M - 1} \frac{a_m(\nu)}{z^m}  + R_M^{(K)} (z,\nu ) \right)
\end{equation}
and
\begin{equation}\label{eq61}
K'_\nu(z) =  - \left( \frac{\pi}{2z} \right)^{\frac{1}{2}} e^{ - z} \left( \sum\limits_{m = 0}^{M - 1} \frac{b_m(\nu)}{z^m}  + R_M^{(K')} (z,\nu ) \right).
\end{equation}
Here the coefficients $a_m(\nu)$ and $b_m(\nu)$ are polynomials in $\nu^2$ of degree $m$ (cf. \cite[\S 10.40(i)]{NIST}). A thorough analysis of these expansions has been given recently by the present author \cite{Nemes2}.

For $|\arg z|<\pi$, complex $\mu$ and $\nu$, and any non-negative integers $N$ and $M$, we define the remainder terms $R_{N,M}^{(S)} (z,\mu ,\nu )$ and $R_{N,M}^{(S')} (z,\mu ,\nu )$ by
\begin{gather}\label{eq23}
\begin{split}
R_N^{(S)} (z,\mu ,\nu ) = \; & ( - 1)^N \frac{2^{\mu  + \frac{1}{2}} \pi ^{\frac{1}{2}} }{\Gamma \left( \frac{ - \mu  + \nu  + 1}{2} \right)\Gamma \left( \frac{ - \mu  - \nu  + 1}{2} \right)}\frac{1}{z^{2N}} \\ & \times \left( \sum\limits_{m = 0}^{M - 1} a_m (\nu )\Gamma \left( 2N - m - \mu  + \frac{1}{2} \right)\Pi _{2N - m - \mu  + \frac{1}{2}} (z)  + R_{N,M}^{(S)} (z,\mu ,\nu ) \right),
\end{split}
\end{gather}
provided $M < 2N - \Re (\mu ) + \frac{1}{2}$, and
\begin{gather}\label{eq24}
\begin{split}
R_N^{(S')} (z,\mu ,\nu ) = \; & ( - 1)^{N + 1} \frac{2^{\mu  + \frac{1}{2}} \pi ^{\frac{1}{2}} }{\Gamma \left( \frac{ - \mu  + \nu  + 1}{2} \right)\Gamma \left( \frac{ - \mu  - \nu  + 1}{2} \right)}\frac{1}{z^{2N}}  \\ & \times \left( \sum\limits_{m = 0}^{M - 1} b_m (\nu )\Gamma \left( 2N - m - \mu  + \frac{3}{2} \right)\Pi _{2N - m - \mu  + \frac{3}{2}} (z)  + R_{N,M}^{(S')} (z,\mu ,\nu ) \right) ,
\end{split}
\end{gather}
provided $M < 2N - \Re (\mu ) + \frac{3}{2}$. The re-expansion \eqref{eq23} is equivalent to that studied in \cite{Nemes1}.

In the following theorem, we give bounds for the remainders $R_{N,M}^{(S)} (z,\mu ,\nu )$ and $R_{N,M}^{(S')} (z,\mu ,\nu )$.
\begin{theorem}\label{thm7}
Let $N$ and $M$ be arbitrary non-negative integers, and let $\mu$ and $\nu$ be complex numbers satisfying $\Re(\mu)<2N-M+\frac{1}{2}$ and $|\Re(\nu)|<M+\frac{1}{2}$. Then we have
\begin{gather}\label{eq55}
\begin{split}
\big| R_{N,M}^{(S)} (z,\mu ,\nu ) \big| \le \; & \left| z \right|^M \big| R_M^{(K)} (\left| z \right|,\nu ) \big|\left| \Gamma \left( 2N - M - \mu  + \frac{1}{2} \right) \right|\big| \Pi _{2N - M - \mu  + \frac{1}{2}} (z) \big|
\\ & + \frac{\left| \cos (\pi \nu ) \right|}{\left| \cos (\pi \Re (\nu )) \right|}\left| a_M (\Re (\nu )) \right|\Gamma \left( 2N - M - \Re (\mu ) + \frac{1}{2} \right),
\end{split}
\end{gather}
for $|\arg z|\leq \frac{\pi}{2}$. Similarly, let $N$ be an arbitrary non-negative integer and $M$ be an arbitrary positive integer, and let $\mu$ and $\nu$ be complex numbers satisfying $\Re(\mu)<2N-M+\frac{3}{2}$ and $|\Re(\nu)|<M-\frac{1}{2}$. Then we have
\begin{gather}\label{eq59}
\begin{split}
\big| R_{N,M}^{(S')} (z,\mu ,\nu ) \big| \le \; & \left| z \right|^M \big| R_M^{(K')} (\left| z \right|,\nu ) \big|\left| \Gamma \left( 2N - M - \mu  + \frac{3}{2} \right) \right|\big| \Pi _{2N - M - \mu  + \frac{3}{2}} (z) \big|
\\ & + \frac{\left| \cos (\pi \nu ) \right|}{\left| \cos (\pi \Re (\nu )) \right|}\left| b_M (\Re (\nu )) \right|\Gamma \left( 2N - M - \Re (\mu ) + \frac{3}{2} \right),
\end{split}
\end{gather}
for $|\arg z|\leq \frac{\pi}{2}$.
\end{theorem}

These error bounds may be further simplified by applying the inequalities $\left| z \right|^M \big| R_M^{(K)} (\left| z \right|,\nu ) \big| \leq \frac{\left| \cos (\pi \nu ) \right|}{\left| \cos (\pi \Re (\nu )) \right|}\left| a_M (\Re (\nu )) \right|$ and $\left| z \right|^M \big| R_M^{(K')} (\left| z \right|,\nu ) \big| \leq \frac{\left| \cos (\pi \nu ) \right|}{\left| \cos (\pi \Re (\nu )) \right|}\left| b_M (\Re (\nu )) \right|$ (see, e.g., \cite{Nemes2}). A brief discussion on the sharpness of the error bound \eqref{eq55} is given in Section \ref{section6}.

In the case when $z$ is positive and both $\mu$ and $\nu$ are real, we shall show that the remainder term $R_{N,M}^{(S)} (z,\mu ,\nu )$ (respectively $R_{N,M}^{(S')} (z,\mu ,\nu )$) does not exceed the corresponding first neglected term in absolute value and has the same sign provided that $\mu <2N-M+\frac{1}{2}$ and $|\nu|<M+\frac{1}{2}$ (respectively $\mu < 2N-M+\frac{3}{2}$ and $|\nu|<M-\frac{1}{2}$). More precisely, we will prove that the following theorem holds.

\begin{theorem}\label{thm8}
Let $N$ and $M$ be arbitrary non-negative integers. Let $z$ be a positive real number and $\mu$ and $\nu$ be arbitrary real numbers such that $\mu <2N-M+\frac{1}{2}$ and $|\nu|<M+\frac{1}{2}$. Then
\begin{equation}\label{eq62}
R_{N,M}^{(S)} (z,\mu ,\nu ) = a_M (\nu )\Gamma \left( 2N - M - \mu  + \frac{1}{2} \right)\Pi _{2N - M - \mu  + \frac{1}{2}} (z)\Theta _{N,M}^{(S)} (z,\mu ,\nu ),
\end{equation}
where $0<\Theta _{N,M}^{(S)} (z,\mu ,\nu )<1$ is an appropriate number that depends on $z$, $\mu$, $\nu$, $N$ and $M$. Similarly, let $N$ be an arbitrary non-negative integer and $M$ be an arbitrary positive integer. Let $z$ be a positive real number and $\mu$ and $\nu$ be arbitrary real numbers such that $\mu <2N-M+\frac{3}{2}$ and $|\nu|<M-\frac{1}{2}$. Then
\begin{equation}\label{eq63}
R_{N,M}^{(S')} (z,\mu ,\nu ) = b_M (\nu )\Gamma \left( 2N - M - \mu  + \frac{3}{2} \right)\Pi _{2N - M - \mu  + \frac{3}{2}} (z)\Theta _{N,M}^{(S')} (z,\mu ,\nu )
\end{equation}
where $0<\Theta _{N,M}^{(S')} (z,\mu ,\nu )<1$ is an appropriate number that depends on $z$, $\mu$, $\nu$, $N$ and $M$.
\end{theorem}

The best accuracy available from the re-expansions \eqref{eq23} and \eqref{eq24} can be obtained by setting $N = \frac{3}{2}|z|+\rho$ and $M = 2|z|+\sigma$, with $\rho$ and $\sigma$ being appropriate (fixed) real numbers. With these choices of $N$ and $M$, Theorem \ref{thm7}, combined with Stirling's approximation for the gamma function and the bound \eqref{eq87}, imply that
\begin{equation}\label{eq25}
z^{-2N} R_{N,M}^{(S)} (z,\mu ,\nu ) = \mathcal{O}_{\mu ,\nu ,\rho ,\sigma } (\left| z \right|^{ - \Re (\mu )} e^{ - 3\left| z \right|} )
\end{equation}
and
\[
z^{-2N} R_{N,M}^{(S')} (z,\mu ,\nu ) = \mathcal{O}_{\mu ,\nu ,\rho ,\sigma } (\left| z \right|^{1 - \Re (\mu )} e^{ - 3\left| z \right|} )
\]
as $z\to \infty$ in the sector $|\arg z|\leq \frac{\pi}{2}$. (Throughout this paper, we use subscripts in the $\mathcal{O}$ notations to indicate the dependence of the implied constant on certain parameters.) The estimate \eqref{eq25} can alternatively be deduced from the hyperasymptotic theory of inhomogeneous second order linear differential equations, discussed in the paper \cite{AODCH}.

The remaining part of the paper is structured as follows. In Section \ref{section2}, we prove the representations for the remainder terms stated in Theorems \ref{thm1}--\ref{thm3}. In Section \ref{section3}, we prove the error bounds given in Theorems \ref{thm4}--\ref{thm6}. Section \ref{section4} discusses the proof of the bounds for the remainders of the re-expansions stated in Theorems \ref{thm7}--\ref{thm8}. Section \ref{section5} provides applications of our results to the
asymptotic expansions of the Anger--Weber-type functions, the Scorer functions, the Struve functions and their derivatives. The paper concludes with a discussion in Section \ref{section6}.

\section{Proof of the representations for the remainder terms}\label{section2}

In this section, we prove the representations for the remainder terms stated in Theorems \ref{thm1}--\ref{thm3}. In order to prove Theorems \ref{thm1} and \ref{thm2}, we require certain estimates for the regularized hypergeometric function given in the following lemma.

\begin{lemma}\label{lemma1} Let $\nu$ and $\lambda$ be arbitrary fixed complex numbers such that $\Re(\lambda) > 0$ or $\lambda=0$. Then
\[
{\bf F}\left( \nu  + \frac{1}{2}, - \nu  + \frac{1}{2};\lambda ; - \frac{t}{2} \right) =  \begin{cases} \mathcal{O}_{\nu ,\lambda } \left( 1 \right) & \text{if } \; \Re ( \lambda  )> 0, \\  \mathcal{O}_{\nu } (t) & \text{if } \; \lambda= 0\end{cases}
\]
as $t\to 0+$, and
\[
{\bf F}\left( \nu  + \frac{1}{2}, - \nu  + \frac{1}{2};\lambda ; - \frac{t}{2} \right) = \mathcal{O}_{\nu ,\lambda } \big( \left( 1 + t \right)^{\left| {\Re \left( \nu  \right)} \right| - \frac{1}{2}} \log \left( 1 + t \right) \big)
\]
as $t\to +\infty$.
\end{lemma}

\begin{proof}
The proof follows from a similar result of the present author \cite[Lemma 2.1]{Nemes2}.
\end{proof}

We continue with the proof of Theorem \ref{thm1}. The modified Bessel function can be represented in terms of the regularized hypergeometric function as
\[
K_\nu  (u) = \left( \frac{\pi }{2u} \right)^{\frac{1}{2}} e^{ - u} u^\lambda  \int_0^{ + \infty } e^{ - ut} t^{\lambda  - 1} {\bf F}\left( \nu  + \frac{1}{2}, - \nu  + \frac{1}{2};\lambda ; - \frac{t}{2} \right)dt,
\]
for $|\arg u|<\frac{\pi}{2}$ and with an arbitrary complex $\lambda$, $\Re (\lambda) > 0$ \cite[ent. 19.2, p. 195]{Oberhettinger}. Inserting this representation into \eqref{eq3} and changing the order of integration, we deduce
\begin{align*}
R_N^{(S)} (z,\mu ,\nu ) = \; & ( - 1)^N \frac{2^{\mu  + \frac{1}{2}} \pi ^{\frac{1}{2}} }{\Gamma \left( \frac{ - \mu  + \nu  + 1}{2} \right)\Gamma \left( \frac{ - \mu  - \nu  + 1}{2} \right)}\frac{1}{z^{2N}} \\ & \times \int_0^{ + \infty } t^{\lambda  - 1} {\bf F}\left( \nu  + \frac{1}{2}, - \nu  + \frac{1}{2};\lambda ; - \frac{t}{2} \right)\int_0^{ + \infty } \frac{u^{2N - \mu  + \lambda  - \frac{1}{2}} e^{ - (1 + t)u} }{1 + (u/z)^2 }du dt .
\end{align*}
The change in the order of integration is justified because the infinite double integrals are absolutely convergent, which can be seen by appealing to Lemma \ref{lemma1}. By the assumption $\Re (\mu ) < 2N + \Re (\lambda ) + \frac{1}{2}$, the inner integral is expressible in terms of the basic terminant $\Pi_p (w)$ since
\[
\int_0^{ + \infty } \frac{u^{2N - \mu  + \lambda  - \frac{1}{2}} e^{ - (1 + t)u} }{1 + (u/z)^2}du  = \frac{\Gamma \left( 2N - \mu  + \lambda  + \frac{1}{2} \right)}{(1 + t)^{2N - \mu  + \lambda  + \frac{1}{2}}}\Pi _{2N - \mu  + \lambda  + \frac{1}{2}} (z(1 + t))
\]
(cf. equation \eqref{eq01}), and therefore we obtain
\begin{align*}
R_N^{(S)} (z,\mu ,\nu ) = \; & ( - 1)^N \frac{2^{\mu  + \frac{1}{2}} \pi ^{\frac{1}{2}} \Gamma \left( 2N - \mu  + \lambda  + \frac{1}{2} \right)}{\Gamma \left( \frac{ - \mu  + \nu  + 1}{2} \right)\Gamma \left( \frac{ - \mu  - \nu  + 1}{2} \right)}\frac{1}{z^{2N}} \\ & \times \int_0^{ + \infty } \frac{t^{\lambda  - 1} }{(1 + t)^{2N - \mu  + \lambda  + \frac{1}{2}} }{\bf F}\left( \nu  + \frac{1}{2}, - \nu  + \frac{1}{2};\lambda ; - \frac{t}{2} \right)\Pi _{2N - \mu  + \lambda  + \frac{1}{2}} (z(1 + t))dt .
\end{align*}
Since $\Pi_p (w) = \mathcal{O}(1)$ as $w \to \infty$ in the sector $|\arg w|< \pi$, by analytic continuation, this representation is valid in a wider range than \eqref{eq3}, namely in $|\arg z| <\pi$.

Let us now turn our attention to the proof of Theorem \ref{thm2}. It is enough to show that
\begin{align*}
& \mathop {\lim }\limits_{\lambda  \to 0 + } \int_0^{ + \infty } \frac{t^{\lambda  - 1} }{(1 + t)^{2N - \mu  + \lambda  + \frac{1}{2}} }{\bf F}\left( \nu  + \frac{1}{2}, - \nu  + \frac{1}{2};\lambda ; - \frac{t}{2} \right)\Pi _{2N - \mu  + \lambda  + \frac{1}{2}} (z(1 + t))dt 
\\ & = \Pi _{2N - \mu  + \frac{1}{2}} (z) + \int_0^{ + \infty } \frac{t^{ - 1} }{(1 + t)^{2N - \mu  + \frac{1}{2}} }{\bf F}\left( \nu  + \frac{1}{2}, - \nu  + \frac{1}{2};0; - \frac{t}{2} \right)\Pi _{2N - \mu  + \frac{1}{2}} (z(1 + t))dt ,
\end{align*}
for any non-negative integer $N$ and fixed complex numbers $z$, $\mu$ and $\nu$ satisfying $|\arg z| < \pi$, $\Re(\mu)+|\Re(\nu)| < 2N +1$ and $\Re (\mu ) < 2N + \frac{1}{2}$. Note that the convergence of the integral in the second line is guaranteed by Lemma \ref{lemma1} and the boundedness of the basic terminant. By a simple algebraic manipulation and an application of the beta integral (cf. \cite[eq. 5.12.3]{NIST}), we can assert that
\begin{align*}
& \int_0^{ + \infty } \frac{t^{\lambda  - 1} }{(1 + t)^{2N - \mu  + \lambda  + \frac{1}{2}} }{\bf F}\left( \nu  + \frac{1}{2}, - \nu  + \frac{1}{2};\lambda ; - \frac{t}{2} \right)\Pi _{2N - \mu  + \lambda  + \frac{1}{2}} (z(1 + t))dt 
\\  = \; & \int_0^{ + \infty } \frac{t^{\lambda  - 1} }{(1 + t)^{2N - \mu  + \lambda  + \frac{1}{2}} }\frac{1}{\Gamma (\lambda )}\Pi _{2N - \mu  + \lambda  + \frac{1}{2}} (z)dt 
\\ & + \int_0^{ + \infty } \frac{t^{\lambda  - 1} }{(1 + t)^{2N - \mu  + \lambda  + \frac{1}{2}}}\left( {\bf F}\left( \nu  + \frac{1}{2}, - \nu  + \frac{1}{2};\lambda ; - \frac{t}{2} \right)\Pi _{2N - \mu  + \lambda  + \frac{1}{2}} (z(1 + t)) \right. \\ & \left.- \frac{1}{\Gamma (\lambda )}\Pi _{2N - \mu  + \lambda  + \frac{1}{2}} (z) \right)dt
\\  = \; & \frac{\Gamma \left( 2N - \mu  + \frac{1}{2} \right)}{\Gamma \left( 2N - \mu  + \lambda  + \frac{1}{2} \right)}\Pi _{2N - \mu  + \lambda  + \frac{1}{2}} (z)
\\ & + \int_0^{ + \infty } \frac{t^{\lambda  - 1} }{(1 + t)^{2N - \mu  + \lambda  + \frac{1}{2}}}\left( {\bf F}\left( \nu  + \frac{1}{2}, - \nu  + \frac{1}{2};\lambda ; - \frac{t}{2} \right)\Pi _{2N - \mu  + \lambda  + \frac{1}{2}} (z(1 + t)) \right. \\ & \left. - \frac{1}{\Gamma (\lambda )}\Pi _{2N - \mu  + \lambda  + \frac{1}{2}} (z) \right)dt .
\end{align*}
By continuity, we have
\[
\mathop {\lim }\limits_{\lambda  \to 0 + } \frac{\Gamma \left( 2N - \mu  + \frac{1}{2} \right)}{\Gamma \left( 2N - \mu  + \lambda  + \frac{1}{2} \right)}\Pi _{2N - \mu  + \lambda  + \frac{1}{2}} (z) = \Pi _{2N - \mu  + \frac{1}{2}} (z),
\]
and thus, it remains to prove that
\begin{gather}\label{eq4}
\begin{split}
\mathop {\lim }\limits_{\lambda  \to 0 + } & \int_0^{ + \infty } \frac{t^{\lambda  - 1} }{(1 + t)^{2N - \mu  + \lambda  + \frac{1}{2}} }\left( {\bf F}\left( \nu  + \frac{1}{2}, - \nu  + \frac{1}{2};\lambda ; - \frac{t}{2} \right)\Pi _{2N - \mu  + \lambda  + \frac{1}{2}} (z(1 + t)) \right. \\ & \left. - \frac{1}{\Gamma (\lambda )}\Pi _{2N - \mu  + \lambda  + \frac{1}{2}} (z) \right)dt 
\\  =\; & \int_0^{ + \infty } \frac{t^{ - 1} }{(1 + t)^{2N - \mu  + \frac{1}{2}} }{\bf F}\left( \nu  + \frac{1}{2}, - \nu  + \frac{1}{2};0; - \frac{t}{2} \right)\Pi _{2N - \mu  + \frac{1}{2}} (z(1 + t))dt .
\end{split}
\end{gather}
We show that the absolute value of the integrand in the first two lines can be bounded pointwise by an absolutely integrable function uniformly with respect to bounded positive values of $\lambda$. Consequently, by the Lebesgue dominated convergence theorem \cite[Theorem 1.4.49, p. 111]{Tao}, the order of the limit and integration can be interchanged and the required result follows. Assume therefore that $\lambda$ is bounded, say $0 < \lambda \ll 1$. Then, by Taylor's theorem, we have
\[
{\bf F}\left( \nu  + \frac{1}{2}, - \nu  + \frac{1}{2};\lambda ; - \frac{t}{2} \right)\Pi _{2N - \mu  + \lambda  + \frac{1}{2}} (z(1 + t)) = \frac{1}{\Gamma (\lambda )}\Pi _{2N - \mu  + \lambda  + \frac{1}{2}} (z) + \mathcal{O}_{z,\mu ,\nu ,N} (t)
\]
as $t\to 0+$. Consequently, the integrand in the first two lines of \eqref{eq4} is $\mathcal{O}_{z,\mu,\nu,N} \left( {t^\lambda  } \right) = \mathcal{O}_{z,\mu,\nu,N} \left( 1 \right)$ as $t\to 0+$ and, appealing to Lemma \ref{lemma1} and to the boundedness of the basic terminant, it is $\mathcal{O}_{z,\mu,\nu,N} ( ( 1 + t)^{\left| \Re (\nu ) \right| - 2N + \Re (\mu )  - 2} \log ( 1 + t) ) + \mathcal{O}_{z,\mu,\nu,N} ( \left( {1 + t} \right)^{- 2N + \Re (\mu )  - \frac{3}{2}} )$ as $t\to+\infty$. Therefore, there exists a positive constant $C_{z,\mu,\nu,N}$, independent of $t$ and $\lambda$ such that the absolute value of the integrand in the first two lines of \eqref{eq4} is bounded from above by
\begin{equation}\label{eq5}
C_{z,\mu,\nu,N} \big(( 1 + t )^{\left| \Re (\nu ) \right| - 2N + \Re (\mu )  - 2} \log ( 1 + t) + ( 1 + t)^{- 2N + \Re (\mu )  - \frac{3}{2}} \big)
\end{equation}
for any $t>0$. Since $\Re(\mu)+|\Re(\nu)| < 2N +1$ and $\Re (\mu ) < 2N + \frac{1}{2}$, the function \eqref{eq5} is absolutely integrable on the positive $t$-axis and the proof is complete.

The representation \eqref{eq21} follows by combining formula \eqref{eq3} with the functional relation \eqref{eq20} and $-2 K'_\nu(t)= K_{\nu-1}(t)+K_{\nu+1}(t)$ (see, e.g., \cite[eq. 10.29.1]{NIST}).

\section{Proof of the error bounds}\label{section3}

In this section, we prove the bounds for the remainder terms $R_N^{(S)} (z,\mu,\nu)$ and $R_N^{(S')} (z,\mu,\nu)$ given in Theorems \ref{thm4}--\ref{thm6}. To this end, we shall state and prove a series of lemmata.

\begin{lemma}\label{lemma2} If $t$ is positive and $\nu$ and $\lambda$ are real numbers such that $\lambda  > \max \left( 0,\frac{1}{2} - \left| \nu  \right| \right)$, then
\[
{\bf F}\left( \nu  + \frac{1}{2}, - \nu  + \frac{1}{2};\lambda ; - \frac{t}{2} \right) \geq 0.
\]
\end{lemma}

\begin{proof}
See \cite[Lemma 3.1]{Nemes2}.
\end{proof}

We remark that, by direct numerical computation, it can be verified that when $\left| \nu  \right| < \frac{1}{2}$ and $0<\lambda  < \frac{1}{2} - \left| \nu  \right|$, the quantity ${\bf F}\left( \nu  + \frac{1}{2}, - \nu  + \frac{1}{2};\lambda ; - \frac{t}{2} \right)$, as a function of $t>0$, may take both positive and negative values. Therefore, in general, the condition $\lambda  > \frac{1}{2} - \left| \nu  \right|$ in Lemma \ref{lemma2} cannot be weakened.

\begin{lemma}\label{lemma3} If $t$ is positive and $\nu$ is real, then
\[
{\bf F}\left( \nu  - \frac{1}{2}, - \nu  + \frac{3}{2};0; - \frac{t}{2} \right) + {\bf F}\left( \nu  + \frac{3}{2}, - \nu  - \frac{1}{2};0; - \frac{t}{2} \right)  \geq 0.
\]
\end{lemma}

\begin{proof}
See \cite[Lemma 3.2]{Nemes2}.
\end{proof}

\begin{lemma} Let $N$ be a non-negative integer and let $\mu$, $\nu$ and $\lambda$ be arbitrary complex numbers such that $\Re (\mu ) + \left| {\Re (\nu )} \right| < 2N + 1$ and $\Re (\lambda ) > 0$. Then
\begin{equation}\label{eq30}
a_N ( - \mu ,\nu ) = \frac{2^{\mu  + \frac{1}{2}} \pi ^{\frac{1}{2}} \Gamma \left( 2N - \mu  + \lambda  + \frac{1}{2} \right)}{\Gamma \left( \frac{- \mu  + \nu  + 1}{2} \right)\Gamma \left( \frac{ - \mu  - \nu  + 1}{2} \right)}\int_0^{ + \infty } \frac{t^{\lambda  - 1} }{(1 + t)^{2N - \mu  + \lambda  + \frac{1}{2}}} {\bf F}\left( \nu  + \frac{1}{2}, - \nu  + \frac{1}{2};\lambda ; - \frac{t}{2} \right)dt ,
\end{equation}
provided $\Re (\mu ) < 2N + \Re (\lambda ) + \frac{1}{2}$, and
\begin{gather}\label{eq31}
\begin{split}
b_N ( - \mu ,\nu ) =  - \frac{2^{\mu  - \frac{1}{2}} \pi ^{\frac{1}{2}} \Gamma \left( 2N - \mu  + \lambda  + \frac{3}{2} \right)}{\Gamma \left( \frac{ - \mu  + \nu  + 1}{2} \right)\Gamma \left( \frac{ - \mu  - \nu  + 1}{2} \right)} & \int_0^{ + \infty } \frac{t^{\lambda  - 1} }{(1 + t)^{2N - \mu  + \lambda  + \frac{3}{2}} }\left( {\bf F}\left( \nu  - \frac{1}{2}, - \nu  + \frac{3}{2};\lambda ; - \frac{t}{2} \right) \right.
\\ & \left. + \; {\bf F}\left( \nu  + \frac{3}{2}, - \nu  - \frac{1}{2};\lambda ; - \frac{t}{2} \right) \right)dt ,
\end{split}
\end{gather}
\begin{gather}\label{eq32}
\begin{split}
b_N ( - \mu ,\nu ) =  - \frac{2^{\mu  - \frac{1}{2}} \pi ^{\frac{1}{2}} \Gamma \left( 2N - \mu  + \frac{3}{2} \right)}{\Gamma \left( \frac{{ - \mu  + \nu  + 1}}{2} \right)\Gamma \left( \frac{ - \mu  - \nu  + 1}{2} \right)}& \left( 2  + \int_0^{ + \infty } \frac{t^{ - 1} }{(1 + t)^{2N - \mu  + \frac{3}{2}} }\left( {\bf F}\left( \nu  - \frac{1}{2}, - \nu  + \frac{3}{2};0; - \frac{t}{2} \right) \right.\right.
\\ & \left.\left. + \; {\bf F}\left( \nu  + \frac{3}{2}, - \nu  - \frac{1}{2};0; - \frac{t}{2} \right) \right)dt  \right) .
\end{split}
\end{gather}
\end{lemma}

\begin{proof} From \eqref{eq22}, we can assert that
\[
a_N ( - \mu ,\nu ) = ( - 1)^N z^{2N} (R_{N + 1}^{(S)} (z,\mu ,\nu ) - R_N^{(S)} (z,\mu ,\nu ))
\]
for any $N \geq 0$. Substituting the integral representation \eqref{eq6} into the right-hand side of this equality and employing the relation $p(p + 1)\Pi _{p + 2} (w) = w^2 (1 - \Pi _p (w))$ (cf. \cite[eq. (44), p. 415]{Dingle}), we arrive at \eqref{eq30}.

In order to prove \eqref{eq31}, we substitute \eqref{eq30} into the right-hand side of the relation
\[
2b_N ( - \mu ,\nu ) = (\mu  + \nu  - 1)a_N ( - \mu  + 1,\nu  - 1) + (\mu  - \nu  - 1)a_N ( - \mu  + 1,\nu  + 1),
\]
which itself is a consequence of the connection formula $2S'_{\mu ,\nu } (z) = (\mu  + \nu  - 1)S_{\mu  - 1,\nu  - 1} (z) + (\mu  - \nu  - 1)S_{\mu  - 1,\nu  + 1} (z)$.

Formula \eqref{eq32} follows from \eqref{eq33} by an argument similar to the derivation of \eqref{eq31} from \eqref{eq6}; the details are left to the reader.
\end{proof}

We continue with the proof of the bound \eqref{eq60}. Let $N$ be a non-negative integer and let $\mu$, $\nu$ and $\lambda$ be arbitrary real numbers such that $\mu+|\nu| < 2N +1$, $\mu < 2N +\lambda+ \frac{1}{2}$ and $\lambda  \ge \max \left( 0,\frac{1}{2} - \left| \nu  \right| \right)$, respectively. From \eqref{eq6}, using Lemma \ref{lemma2}, we infer that
\begin{align*}
\big| R_N^{(S)} (z,\mu ,\nu ) \big| \le \; & \frac{2^{\mu  + \frac{1}{2}} \pi ^{\frac{1}{2}} \Gamma \left( 2N - \mu  + \lambda  + \frac{1}{2} \right)}{\left| \Gamma \left( \frac{ - \mu  + \nu  + 1}{2} \right)\Gamma \left( \frac{ - \mu  - \nu  + 1}{2} \right) \right|} \frac{1}{|z|^{2N}} \\ & \times \int_0^{ + \infty } \frac{t^{\lambda  - 1} }{(1 + t)^{2N - \mu  + \lambda  + \frac{1}{2}} }{\bf F}\left( \nu  + \frac{1}{2}, - \nu  + \frac{1}{2};\lambda ; - \frac{t}{2} \right)dt \mathop {\sup }\limits_{r \ge 1} \big| {\Pi _{2N - \mu  + \lambda  + \frac{1}{2}} (zr)} \big|,
\end{align*}
as long as $|\arg z|<\pi$. The right-hand side can be simplified using Lemma \ref{lemma2} and the representation \eqref{eq30} for the coefficients $a_N ( - \mu ,\nu )$, and we deduce the required bound \eqref{eq60}.

The bound \eqref{eq27} can be proved as follows. Assume that $N$ is a non-negative integer and $\mu$, $\nu$ are arbitrary real numbers such that $\mu+|\nu| < 2N +1$. By substituting \eqref{eq33} into the functional relation \eqref{eq20}, we find
\begin{gather}\label{eq40}
\begin{split}
R_N^{(S')} (z,\mu ,\nu ) = \; & - \frac{2^{\mu  - \frac{1}{2}} \pi ^{\frac{1}{2}} \Gamma \left( 2N - \mu  + \frac{3}{2} \right)}{\Gamma \left( \frac{{ - \mu  + \nu  + 1}}{2} \right)\Gamma \left( \frac{ - \mu  - \nu  + 1}{2} \right)} \frac{1}{z^{2N}}\\ & \times \left( 2\Pi _{2N - \mu  + \frac{3}{2}} (z)  + \int_0^{ + \infty } \frac{t^{ - 1} }{(1 + t)^{2N - \mu  + \frac{3}{2}} }\left( {\bf F}\left( \nu  - \frac{1}{2}, - \nu  + \frac{3}{2};0; - \frac{t}{2} \right) \right.\right.
\\ &  \left.\left. +\; {\bf F}\left( \nu  + \frac{3}{2}, - \nu  - \frac{1}{2};0; - \frac{t}{2} \right) \right)\Pi _{2N - \mu  + \frac{3}{2}} (z(1 + t)) dt  \right) .
\end{split}
\end{gather}
By Lemma \ref{lemma3}, we can then assert that
\begin{align*}
\big| R_N^{(S')} (z,\mu ,\nu ) \big| \leq \; & \frac{2^{\mu  - \frac{1}{2}} \pi ^{\frac{1}{2}} \Gamma \left( 2N - \mu  + \frac{3}{2} \right)}{\left|\Gamma \left( \frac{{ - \mu  + \nu  + 1}}{2} \right)\Gamma \left( \frac{ - \mu  - \nu  + 1}{2} \right)\right|} \frac{1}{|z|^{2N}}\\ & \times \left( 2  + \int_0^{ + \infty } \frac{t^{ - 1} }{(1 + t)^{2N - \mu  + \frac{3}{2}} }\left( {\bf F}\left( \nu  - \frac{1}{2}, - \nu  + \frac{3}{2};0; - \frac{t}{2} \right) \right.\right.
\\ &  \left.\left. +\; {\bf F}\left( \nu  + \frac{3}{2}, - \nu  - \frac{1}{2};0; - \frac{t}{2} \right) \right) dt  \right) \mathop {\sup }\limits_{r \ge 1} \big| {\Pi _{2N - \mu  + \frac{3}{2}} (zr)} \big|,
\end{align*}
provided $|\arg z|<\pi$. The right-hand side of this inequality can be simplified using Lemma \ref{lemma3} and the representation \eqref{eq32} for the coefficients $b_N ( - \mu ,\nu )$ leading to the required estimate \eqref{eq27}.

To prove formula \eqref{eq41}, we first note that $0<\Pi_p(w) < 1$ whenever $w>0$ and $p >0$ (see Proposition \ref{prop1}). Employing Lemma \ref{lemma3}, the mean value theorem of integration and the fact that the terms in the large parentheses in \eqref{eq40} are all non-negative, we can infer that
\begin{align*}
R_N^{(S')} (z,\mu ,\nu ) = \; & - \frac{2^{\mu  - \frac{1}{2}} \pi ^{\frac{1}{2}} \Gamma \left( 2N - \mu  + \frac{3}{2} \right)}{\Gamma \left( \frac{{ - \mu  + \nu  + 1}}{2} \right)\Gamma \left( \frac{ - \mu  - \nu  + 1}{2} \right)} \frac{1}{z^{2N}}\\ & \times \left( 2  + \int_0^{ + \infty } \frac{t^{ - 1} }{(1 + t)^{2N - \mu  + \frac{3}{2}} }\left( {\bf F}\left( \nu  - \frac{1}{2}, - \nu  + \frac{3}{2};0; - \frac{t}{2} \right) \right.\right.
\\ &  \left.\left. +\; {\bf F}\left( \nu  + \frac{3}{2}, - \nu  - \frac{1}{2};0; - \frac{t}{2} \right) \right) dt  \right) \theta_N^{(S')}(z,\mu,\nu),
\end{align*}
for some $0<\theta_N^{(S')}(z,\mu,\nu)<1$ depending on $z$, $\mu$, $\nu$ and $N$. Comparing this expression with the representation \eqref{eq32} for the coefficients $b_N ( - \mu ,\nu )$ yields the required formula \eqref{eq41} for $R_N^{(S')} (z,\mu ,\nu )$.

We now turn to the proof of the bound \eqref{eq43}. Assume that $N$ is a non-negative integer and $\mu$ and $\nu$ are arbitrary complex numbers satisfying $\Re(\mu)+|\Re(\nu)| < 2N +1$. First note that since
\[
{\bf F}\left( \nu  + \frac{1}{2}, - \nu  + \frac{1}{2};\frac{1}{2}; - s \right) = \frac{(\left( 1 + s \right)^{\frac{1}{2}}  + s^{\frac{1}{2}} )^{2\nu }  + (\left( 1 + s \right)^{\frac{1}{2}}  - s^{\frac{1}{2}} )^{2\nu } }{2\left( 1 + s \right)^{\frac{1}{2}} }
\]
for any $s>0$ (see, for instance, \cite[eq. 15.4.13]{NIST}), the following inequality holds:
\begin{equation}\label{eq03}
\left| {\bf F}\left( \nu  + \frac{1}{2}, - \nu  + \frac{1}{2};\frac{1}{2}; - s \right) \right| \le {\bf F}\left( \Re \left( \nu  \right) + \frac{1}{2}, - \Re \left( \nu  \right) + \frac{1}{2};\frac{1}{2}; - s \right).
\end{equation}
Taking $\lambda=\frac{1}{2}$ in \eqref{eq6} and employing \eqref{eq03} with $\frac{t}{2}$ in place of $s$, we find that
\begin{align*}
& \big| R_N^{(S)} (z,\mu ,\nu ) \big| \le \frac{2^{\Re(\mu)  + \frac{1}{2}} \pi ^{\frac{1}{2}} |\Gamma ( 2N - \mu  + 1 )|}{\left| \Gamma \left( \frac{ - \mu  + \nu  + 1}{2} \right)\Gamma \left( \frac{ - \mu  - \nu  + 1}{2} \right) \right|} \frac{1}{|z|^{2N}} \\ & \times \int_0^{ + \infty } \frac{t^{-\frac{1}{2}} }{(1 + t)^{2N - \Re(\mu)  + 1} }{\bf F}\left( \Re(\nu)  + \frac{1}{2}, - \Re(\nu)  + \frac{1}{2};\frac{1}{2} ; - \frac{t}{2} \right)dt \mathop {\sup }\limits_{r \ge 1} \big| {\Pi _{2N - \mu  + 1} (zr)} \big|,
\end{align*}
as long as $|\arg z|<\pi$. Now the required bound \eqref{eq43} can be obtained by simplifying the right-hand side of the above inequality using the representation \eqref{eq30} (with $\lambda=\frac{1}{2}$) for the coefficients $a_N ( - \mu ,\nu )$ (with $\Re(\mu)$ and $\Re(\nu)$ in place of $\mu$ and $\nu$).

We close this section by proving the estimate \eqref{eq44}. For this purpose, suppose that $N$ is a non-negative integer and $\mu$ and $\nu$ are arbitrary complex numbers such that $\Re(\mu)+|\Re(\nu)| < 2N +1$. By taking $\lambda=\frac{1}{2}$ in \eqref{eq6} and substituting the result into the functional relation \eqref{eq20}, we obtain
\begin{align*}
R_N^{(S')} (z,\mu ,\nu ) = \; & - \frac{2^{\mu  - \frac{1}{2}} \pi ^{\frac{1}{2}} \Gamma ( 2N - \mu  + 2)}{\Gamma \left( \frac{{ - \mu  + \nu  + 1}}{2} \right)\Gamma \left( \frac{ - \mu  - \nu  + 1}{2} \right)} \frac{1}{z^{2N}}  \\ & \times \int_0^{ + \infty } \frac{t^{ - \frac{1}{2}} }{(1 + t)^{2N - \mu  + 2} }\left( {\bf F}\left( \nu  - \frac{1}{2}, - \nu  + \frac{3}{2};\frac{1}{2}; - \frac{t}{2} \right) \right.
\\ &  \left. +\; {\bf F}\left( \nu  + \frac{3}{2}, - \nu  - \frac{1}{2};\frac{1}{2}; - \frac{t}{2} \right) \right)\Pi _{2N - \mu  + 2} (z(1 + t)) dt.
\end{align*}
An application of the inequality \eqref{eq03} yields
\begin{align*}
\big| R_N^{(S')} (z,\mu ,\nu ) \big| \leq \; & \frac{2^{\Re(\mu)  - \frac{1}{2}} \pi ^{\frac{1}{2}} |\Gamma ( 2N - \mu  + 2)|}{\left|\Gamma \left( \frac{{ - \mu  + \nu  + 1}}{2} \right)\Gamma \left( \frac{ - \mu  - \nu  + 1}{2} \right)\right|} \frac{1}{|z|^{2N}} \\ & \times \int_0^{ + \infty } \frac{t^{ - \frac{1}{2}} }{(1 + t)^{2N - \Re(\mu)  + 2} }\left( {\bf F}\left( \Re(\nu)  - \frac{1}{2}, - \Re(\nu)  + \frac{3}{2};\frac{1}{2}; - \frac{t}{2} \right) \right.
\\ &  \left. +\; {\bf F}\left( \Re(\nu)  + \frac{3}{2}, - \Re(\nu)  - \frac{1}{2};\frac{1}{2}; - \frac{t}{2} \right) \right) dt \mathop {\sup }\limits_{r \ge 1} \big| {\Pi _{2N - \mu  + 2} (zr)} \big|,
\end{align*}
for $|\arg z|<\pi$. We then simplify the right-hand side of this inequality using the representation \eqref{eq31} (with $\lambda=\frac{1}{2}$) for the coefficients $b_N ( - \mu ,\nu )$ (with $\Re(\mu)$ and $\Re(\nu)$ in place of $\mu$ and $\nu$) to arrive at the required estimate \eqref{eq44}.

\section{Proof of the error bounds for the re-expansions}\label{section4}

In this section, we prove the bounds for the remainder terms $R_{N,M}^{(S)} (z,\mu ,\nu )$ and $R_{N,M}^{(S')} (z,\mu ,\nu )$ that are given in Theorems \ref{thm7} and \ref{thm8}. We begin by stating the following lemmata, whose proof can be found in \cite{Nemes2}.

\begin{lemma} Let $M$ be a non-negative integer and let $\nu$ be an arbitrary complex number. Then
\begin{equation}\label{eq50}
a_M (\nu) = (-1)^M \left( \frac{2}{\pi} \right)^{\frac{1}{2}} \frac{\cos ( \pi \nu)}{\pi}\int_0^{ + \infty } t^{M - \frac{1}{2}} e^{ - t} K_\nu (t)dt,
\end{equation}
provided $|\Re(\nu)|<M+\frac{1}{2}$, and
\begin{equation}\label{eq51}
b_M (\nu) = (-1)^M \left( \frac{2}{\pi} \right)^{\frac{1}{2}} \frac{\cos (\pi \nu)}{\pi }\int_0^{ + \infty } t^{M - \frac{1}{2}} e^{ - t} K'_\nu (t)dt,
\end{equation}
provided that $|\Re(\nu)|<M-\frac{1}{2}$ and $M \geq 1$.
\end{lemma}

\begin{lemma} Let $M$ be a non-negative integer and let $\nu$ be an arbitrary complex number. Then
\begin{equation}\label{eq52}
R_M^{( K )} ( z,\nu ) = (-1)^M \left( \frac{2}{\pi } \right)^{\frac{1}{2}} \frac{\cos (\pi \nu )}{\pi }\frac{1}{z^M }\int_0^{ + \infty } \frac{t^{M - \frac{1}{2}} e^{ - t} }{1 + t/z}K_\nu  ( t )dt ,
\end{equation}
provided that $\left| {\arg z} \right| < \pi$ and $|\Re(\nu)|<M+\frac{1}{2}$, and
\begin{equation}\label{eq53}
R_M^{(K')} ( z,\nu ) = (- 1)^M \left( \frac{2}{\pi} \right)^{\frac{1}{2}} \frac{\cos ( \pi \nu)}{\pi }\frac{1}{z^M }\int_0^{ + \infty } \frac{t^{M - \frac{1}{2}} e^{ - t} }{1 + t/z}K'_\nu (t)dt ,
\end{equation}
provided that $\left| {\arg z} \right| < \pi$, $|\Re(\nu)|<M-\frac{1}{2}$ and $M \geq 1$.
\end{lemma}

We continue with the proof of the estimate \eqref{eq55}. Let $N$ and $M$ be arbitrary non-negative integers. Suppose that $\Re(\mu)<2N-M+\frac{1}{2}$, $|\Re(\nu)|<M+\frac{1}{2}$ and $| \arg z| < \frac{\pi}{2}$. We begin by replacing the function $K_\nu (t)$ in the integral formula \eqref{eq3} for $R_{N}^{(S)} (z,\mu ,\nu )$ by its truncated asymptotic expansion \eqref{eq54} and using the representation \eqref{eq01} of the basic terminant $\Pi_p(w)$. In this way, we obtain \eqref{eq23} with
\begin{gather}\label{eq57}
\begin{split}
R_{N,M}^{(S)} (z,\mu ,\nu ) & = \int_0^{ + \infty } \frac{t^{2N - \mu  - \frac{1}{2}} e^{ - t} }{1 + (t/z)^2 }R_M^{(K)} (t,\nu )dt \\  & = |z|^{2N - \mu  + \frac{1}{2}} \int_0^{ + \infty } \frac{u^{2N - \mu  - \frac{1}{2}} e^{ - |z|u} }{1 + u^2 e^{ - 2i\arg z} }R_M^{(K)} (|z|u,\nu )du .
\end{split}
\end{gather}
When passing to the second equality, we have made a change of integration variable from $t$ to $u$ by $t = |z|u$. The remainder $R_M^{(K)} (|z|u,\nu )$ is given by the integral formula \eqref{eq52}, which can be re-expressed in the form
\begin{gather}\label{eq58}
\begin{split}
R_M^{(K)} (|z|u ,\nu) = \; & (- 1)^M \left( \frac{2}{\pi} \right)^{\frac{1}{2}} \frac{\cos (\pi \nu )}{\pi}\frac{1}{(|z|u)^M }\int_0^{ + \infty } \frac{t^{M - \frac{1}{2}} e^{ - t} }{1 + t/\left|z\right|}K_\nu (t)dt 
\\ & + ( - 1)^M \left( \frac{2}{\pi } \right)^{\frac{1}{2}} \frac{\cos ( \pi \nu)}{\pi }\frac{u  - 1}{( |z|u )^M}\int_0^{ + \infty } \frac{t^{M - \frac{1}{2}} e^{ - t} }{\left( 1 + |z|u /t \right)\left( 1 + t/|z| \right)}K_\nu ( t )dt 
\\ = \; & u^{ - M} R_M^{( K )} (|z|,\nu ) \\ & + (  - 1)^M \left( \frac{2}{\pi } \right)^{\frac{1}{2}} \frac{\cos ( \pi \nu)}{\pi }\frac{u  - 1}{(|z|u)^M }\int_0^{ + \infty } \frac{t^{M - \frac{1}{2}} e^{ - t} }{\left( 1 + |z|u /t \right)\left( 1 +t/|z| \right)}K_\nu(t)dt .
\end{split}
\end{gather}
Noting that
\[
0 < \frac{1}{\left( 1 + |z|u/t \right)\left( 1 + t/|z| \right)} < 1
\]
for positive $u$ and $t$, the substitution of \eqref{eq58} into \eqref{eq57} and trivial estimation yield the upper bound
\begin{align*}
& \big| R_{N,M}^{(S)} (z,\mu ,\nu ) \big| \le  \big| |z|^{2N - \mu  + \frac{1}{2}} \big|\left| \int_0^{ + \infty } \frac{u^{2N - M - \mu  - \frac{1}{2}} e^{ - \left| z \right|u} }{1 + u^2 e^{ - 2i\arg z} }du \right|\big| R_M^{(K)} (\left| z \right|,\nu ) \big|
\\ & + \frac{\left| \cos (\pi \nu ) \right|}{\left| \cos (\pi \Re (\nu )) \right|}\left| a_M (\Re (\nu )) \right|\big| |z|^{2N - M - \mu  + \frac{1}{2}}  \big|\int_0^{ + \infty } \big| u^{2N - M - \mu  - \frac{1}{2}} \big|e^{ - \left| z \right|u} \left| \frac{u - 1}{u^2 +  e^{ 2i\arg z} } \right|du .
\end{align*}
In arriving at this bound, we have made use of the fact that $|K_\nu(t)|<K_{\Re(\nu)}(t)$ for any $t > 0$\footnote{This inequality follows, for example, from the integral representation $K_\nu(t) = \frac{1}{2}\int_{ - \infty }^{ + \infty } e^{ - t\cosh u - \nu u} du$ (see, e.g., \cite[eq. (7), p. 182]{Watson}).} and of the representation \eqref{eq50} for the coefficients $a_N (\nu)$ (with $\Re(\nu)$ in place of $\nu$). Since $| ( u - 1)/( u^2 + e^{2i\arg z} )| \le 1$ for positive $u$, we find, after simplification, that
\begin{align*}
\big| R_{N,M}^{(S)} (z,\mu ,\nu ) \big| \le \; & \left| z \right|^M \big| R_M^{(K)} (\left| z \right|,\nu ) \big|\left| \Gamma \left( 2N - M - \mu  + \frac{1}{2} \right) \right|\big| \Pi _{2N - M - \mu  + \frac{1}{2}} (z) \big|
\\ & + \frac{\left| \cos (\pi \nu ) \right|}{\left| \cos (\pi \Re (\nu )) \right|}\left| a_M (\Re (\nu )) \right|\Gamma \left( 2N - M - \Re (\mu ) + \frac{1}{2} \right).
\end{align*}
By continuity, this bound holds in the closed sector $\left|\arg z\right| \leq \frac{\pi}{2}$, and therefore the proof of the estimate \eqref{eq55} is complete.

The corresponding bound \eqref{eq59} for the remainder term $R_{N,M}^{(S')} (z,\mu ,\nu )$ can be proved in an analogous manner using the representations \eqref{eq21}, \eqref{eq61}, \eqref{eq51} and \eqref{eq53}. We leave the details to the reader.

We now turn to the proof of the expression \eqref{eq62}. Let $N$ and $M$ be arbitrary non-negative integers. Suppose that $z$ is positive and $\mu$ and $\nu$ are real numbers such that $\mu<2N-M+\frac{1}{2}$ and $|\nu|<M+\frac{1}{2}$. It is known that in the case when $t$ is positive and $\nu$ is real, $|\nu|<M + \frac{1}{2}$, we have
\begin{equation}\label{eq64}
R_M^{(K)} (t,\nu ) = \frac{a_M (\nu )}{t^M}\theta _M^{(K)} (t,\nu ).
\end{equation}
Here $0<\theta _M^{(K)} (t,\nu )<1$ is an appropriate number that depends on $t$, $\nu$ and $M$ (see, for instance, \cite{Nemes2}). Applying this formula in the representation \eqref{eq57} for $R_{N,M}^{(S)} (z,\mu ,\nu )$, and then using the mean value theorem of integration, we find that
\begin{align*}
R_{N,M}^{(S)} (z,\mu ,\nu ) & = a_M (\nu )\int_0^{ + \infty } \frac{t^{2N - M - \mu  - \frac{1}{2}} e^{ - t} }{1 + (t/z)^2}\theta _M^{(K)} (t,\nu )dt 
\\ & = a_M (\nu )\Theta _{N,M}^{(K)} (z,\mu ,\nu )\int_0^{ + \infty } \frac{t^{2N - M - \mu  - \frac{1}{2}} e^{ - t} }{1 + (t/z)^2}dt .
\end{align*}
Here, $0<\Theta _{N,M}^{(K)} (z,\mu ,\nu )<1$ is a suitable number that depends on $z$, $\mu$, $\nu$, $N$ and $M$. The integral in the second line can be expressed in terms of the basic terminant $\Pi_p(w)$ by making use of the formula \eqref{eq01}. Hence the expression \eqref{eq62} follows.

The corresponding result \eqref{eq63} for $R_{N,M}^{(S')} (z,\mu ,\nu )$ can be obtained using the analogues of the representations \eqref{eq57} and \eqref{eq64} (for the latter, see \cite{Nemes2}).

\section{Application to related functions}\label{section5}

In this section, we show how the remainder terms of the known asymptotic expansions of the Anger--Weber-type functions, the Scorer functions, the Struve functions and their derivatives may be expressed in terms of the remainders $R_{N}^{(S)} (z,\mu ,\nu )$ and $R_{N}^{(S')} (z,\mu ,\nu )$. Thus, the theorems in Section \ref{section1} can be applied to obtain representations and bounds for the error terms of these asymptotic expansions. Some of the resulting representations and bounds are well known in the literature but many of them, we believe, are new.

The asymptotic expansions of the Anger function ${\bf J}_\nu  (z)$, the Weber function ${\bf E}_\nu  (z)$ and the associated Anger--Weber function ${\bf A}_\nu  (z)$ may be written
\begin{equation}\label{eq67}
{\bf J}_\nu  (z) = J_\nu  (z) + \frac{\sin (\pi \nu )}{\pi z}\left( \left( \sum\limits_{n = 0}^{N - 1} \frac{F_n (\nu )}{z^{2n} }  + R_N^{(1)} (z,\nu ) \right) - \frac{\nu }{z}\left( \sum\limits_{m = 0}^{M - 1} \frac{G_m (\nu )}{z^{2m}}  + R_M^{(2)} (z,\nu ) \right) \right),
\end{equation}
\begin{gather}
\begin{split}
{\bf E}_\nu  (z) =  - Y_\nu  (z) & - \frac{1 + \cos (\pi \nu )}{\pi z}\left( \sum\limits_{n = 0}^{N - 1} \frac{F_n (\nu )}{z^{2n} }  + R_N^{(1)} (z,\nu ) \right) \\ & - \frac{\nu (1 - \cos (\pi \nu ))}{\pi z^2 }\left( \sum\limits_{m = 0}^{M - 1} \frac{G_m (\nu )}{z^{2m} }  + R_M^{(2)} (z,\nu ) \right)
\end{split}
\end{gather}
and
\begin{equation}\label{eq68}
{\bf A}_\nu  (z) = \frac{1}{\pi z}\left( \sum\limits_{n = 0}^{N - 1} \frac{F_n (\nu )}{z^{2n} }  + R_N^{(1)} (z,\nu ) \right) - \frac{\nu }{\pi z^2}\left( \sum\limits_{m = 0}^{M - 1} \frac{G_m (\nu )}{z^{2m}}  + R_M^{(2)} (z,\nu ) \right).
\end{equation}
Here $N$ and $M$ are arbitrary non-negative integers and $R_N^{(1)} (z,\nu ) = \mathcal{O}_{\nu ,N} (|z|^{ - 2N} )$, $R_M^{(2)} (z,\nu ) = \mathcal{O}_{\nu ,M} (|z|^{ - 2M} )$ as $z\to \infty$ in the sector $| \arg z|\leq \pi-\delta$ with $\delta>0$ being fixed (cf. \cite[\S11.11(i)]{NIST}). The coefficients $F_n (\nu )$ and $G_n (\nu )$ are polynomials in $\nu^2$ of degree $n$ and they are given explicitly by $F_0(\nu)=G_0(\nu)=1$ and
\[
F_n (\nu ) = \prod\limits_{k = 1}^n (\nu ^2  - (2k - 1)^2 )  = ( - 1)^n a_n (0,\nu ),\;\, G_n (\nu )  = \prod\limits_{k = 1}^n (\nu ^2  - (2k)^2 )  = ( - 1)^n a_n (1,\nu ),
\]
for $n\geq1$. The asymptotic expansions of the functions ${\bf J}_\nu  (z)$ and ${\bf E}_\nu  (z)$ were stated without proof in 1879 by Weber \cite{Weber} and in the subsequent year by Lommel \cite{Lommel}. They were proved as special cases of much more general formulae by Nielsen \cite[p. 228]{Nielsen} in 1904 (see also \cite[\S 10.14]{Watson}). The Anger and Weber functions are related to the Lommel function through the functional equations \cite[eqs. (9) and (10), p. 84]{Luke}
\begin{equation}\label{eq65}
{\bf J}_\nu  (z) = J_\nu  (z) + \frac{\sin (\pi \nu )}{\pi }(S_{0,\nu } (z) - \nu S_{ - 1,\nu } (z))
\end{equation}
and
\begin{equation}
{\bf E}_\nu  (z) =  - Y_\nu  (z) - \frac{1 + \cos (\pi \nu )}{\pi }S_{0,\nu } (z) - \frac{\nu (1 - \cos (\pi \nu ))}{\pi }S_{ - 1,\nu } (z).
\end{equation}
Also, by combining \eqref{eq65} with the connection formula ${\bf J}_\nu  (z) = J_\nu  (z) + \sin (\pi \nu ){\bf A}_\nu  (z)$ \cite[eq. 11.10.15]{NIST}, we find
\begin{equation}\label{eq75}
{\bf A}_\nu  (z) = \frac{S_{0,\nu } (z) - \nu S_{ - 1,\nu } (z)}{\pi } .
\end{equation}
If we substitute \eqref{eq22} into the right-hand sides of these equalities and compare the results with \eqref{eq67}--\eqref{eq68}, we can infer that
\begin{equation}\label{eq66}
R_N^{(1)} (z,\nu ) = R_N^{(S)} (z,0,\nu )\; \text{ and } \; R_M^{(2)} (z,\nu ) = R_M^{(S)} (z, - 1,\nu ).
\end{equation}
The known estimates for $R_N^{(1)} (z,\nu )$ and $R_M^{(2)} (z,\nu )$ by Meijer \cite{Meijer} can all be deduced as direct consequences of \eqref{eq66}, \eqref{eq43} and the bounds given in Appendix \ref{appendixa}.

By taking the derivative of both sides of the equalities \eqref{eq67}--\eqref{eq68} with respect to $z$, it is seen that the functions ${\bf J}'_\nu  (z)$, ${\bf E}'_\nu  (z)$ and ${\bf A}'_\nu  (z)$ admit expansions of the form
\begin{gather}\label{eq69}
\begin{split}
{\bf J}'_\nu  (z) = J'_\nu  (z) & - \frac{\sin (\pi \nu )}{\pi z^2 }\left( \left( \sum\limits_{n = 0}^{N - 1} \frac{(2n + 1)F_n (\nu )}{z^{2n} }  + \widetilde R_N^{(1)} (z,\nu ) \right) \right. \\ & \left.- \frac{\nu }{z}\left( \sum\limits_{m = 0}^{M - 1} \frac{(2m + 2)G_m (\nu )}{z^{2m} }  + \widetilde R_M^{(2)} (z,\nu )\right) \right),
\end{split}
\end{gather}
\begin{gather}
\begin{split}
{\bf E}'_\nu  (z) = & - Y'_\nu  (z) + \frac{1 + \cos (\pi \nu )}{\pi z^2 }\left( \sum\limits_{n = 0}^{N - 1} \frac{(2n + 1)F_n (\nu )}{z^{2n} }  + \widetilde R_N^{(1)} (z,\nu ) \right)
\\ & + \frac{\nu (1 - \cos (\pi \nu ))}{\pi z^3}\left( \sum\limits_{m = 0}^{M - 1} \frac{(2m + 2)G_m (\nu )}{z^{2m} }  + \widetilde R_M^{(2)} (z,\nu ) \right)
\end{split}
\end{gather}
and
\begin{equation}\label{eq70}
{\bf A}'_\nu  (z) =  - \frac{1}{\pi z^2}\left( \sum\limits_{n = 0}^{N - 1} \frac{(2n + 1)F_n (\nu )}{z^{2n} }  + \widetilde R_N^{(1)} (z,\nu ) \right) + \frac{\nu }{\pi z^3}\left( \sum\limits_{m = 0}^{M - 1} \frac{(2m + 2)G_m (\nu )}{z^{2m}}  + \widetilde R_M^{(2)} (z,\nu ) \right) .
\end{equation}
If we differentiate each side of the equalities \eqref{eq65}--\eqref{eq75} with respect to $z$, substitute \eqref{eq76} into the right-hand sides of the resulting equalities and compare them with \eqref{eq69}--\eqref{eq70}, we find that
\[
\widetilde R_N^{(1)} (z,\nu ) = -R_N^{(S')} (z,0,\nu )\; \text{ and } \; \widetilde R_M^{(2)} (z,\nu ) = -R_M^{(S')} (z, - 1,\nu ).
\]
In particular, $\widetilde R_N^{(1)} (z,\nu ) = \mathcal{O}_{\nu ,N} (|z|^{ - 2N} )$, $\widetilde R_M^{(2)} (z,\nu ) = \mathcal{O}_{\nu ,M} (|z|^{ - 2M} )$ as $z\to \infty$ in the sector $| \arg z|\leq \pi-\delta$ with any fixed $\delta>0$. Consequently, the expansions \eqref{eq69}--\eqref{eq70} are the large-$z$ asymptotic expansions of the derivatives of the functions ${\bf J}_\nu (z)$, ${\bf E}_\nu (z)$ and ${\bf A}_\nu (z)$.

The first Scorer function $\operatorname{Hi}(z)$ and its derivative have the asymptotic expansions
\begin{equation}\label{eq71}
\operatorname{Hi}(-z) = \frac{1}{\pi z}\left( \sum\limits_{n = 0}^{N - 1} ( - 1)^n \frac{(3n)!}{n!(3z^3 )^n }  + R_N^{(\operatorname{Hi})} (z) \right)
\end{equation}
and
\begin{equation}\label{eq72}
\operatorname{Hi}'( - z) = \frac{1}{\pi z^2}\left( \sum\limits_{n = 0}^{N - 1} ( - 1)^n \frac{(3n + 1)!}{n!(3z^3 )^n}  + R_N^{(\operatorname{Hi}')} (z) \right),
\end{equation}
for any $N \geq 0$, where $R_N^{(\operatorname{Hi})}(z), R_N^{(\operatorname{Hi}')}(z) = \mathcal{O}_N (|z|^{ - 3N} )$ as $z\to \infty$ in the sector $|\arg z|\leq \frac{2\pi}{3}-\delta$, with any fixed $\delta>0$ (see, for instance, \cite[\S9.12(viii)]{NIST}). The asymptotic expansion \eqref{eq71} was established in 1950 by Scorer \cite{Scorer} for large positive values of $z$. Extension to complex $z$ and the corresponding expansion \eqref{eq72} for the derivative is due to Luke \cite[p. 138]{Luke} from 1962. It is known that the Scorer function $\operatorname{Hi}(z)$ can be expressed in terms of the Lommel function as follows \cite[eq. (6), p. 134]{Luke}:
\[
\operatorname{Hi} ( - z) = \frac{2}{3\pi}z^{\frac{1}{2}}S_{0,\frac{1}{3}} (\tfrac{2}{3}z^{\frac{3}{2}}) .
\]
Differentiating both sides of this equality and employing the functional equation $(\nu /z)S_{\mu ,\nu } (z) + S'_{\mu ,\nu } (z)= (\mu  + \nu  - 1)S_{\mu  - 1,\nu  - 1} (z) = (\mu  + \nu  - 1)S_{\mu  - 1,1 - \nu } (z)$ for the Lommel function, we also have
\[
\operatorname{Hi}'( - z) =  - \frac{1}{3\pi }z^{ - \frac{1}{2}} S_{0,\frac{1}{3}} (\tfrac{2}{3}z^{\frac{3}{2}} ) - \frac{2}{3\pi}zS'_{0,\frac{1}{3}} (\tfrac{2}{3}z^{\frac{3}{2}} ) = \frac{4}{9\pi}zS_{ - 1,\frac{2}{3}} (\tfrac{2}{3}z^{\frac{3}{2}} ) .
\]
Substituting \eqref{eq22} into the right-hand sides of these equalities and comparing the results with \eqref{eq71} and \eqref{eq72}, we deduce that
\begin{equation}\label{eq77}
R_N^{(\operatorname{Hi})} (z) =  R_N^{(S)} ( \tfrac{2}{3}z^{\frac{3}{2}} ,0,\tfrac{1}{3} )  \; \text{ and } \;  R_N^{(\operatorname{Hi}')} (z) = R_N^{(S)} (\tfrac{2}{3}z^{\frac{3}{2}} , - 1,\tfrac{2}{3}) .
\end{equation}

The second Scorer function $\operatorname{Gi}(z)$ and its derivative have the asymptotic expansions
\begin{equation}\label{eq73}
\operatorname{Gi}(z) = \frac{1}{\pi z}\left( \sum\limits_{n = 0}^{N - 1} \frac{(3n)!}{n!(3z^3 )^n }  + R_N^{(\operatorname{Gi})} (z) \right)
\end{equation}
and
\begin{equation}\label{eq74}
\operatorname{Gi}'(z) =  - \frac{1}{{\pi z^2 }}\left( \sum\limits_{n = 0}^{N - 1} \frac{(3n + 1)!}{n!(3z^3 )^n}  + R_N^{(\operatorname{Gi}')} (z) \right),
\end{equation}
where $N$ is any non-negative integer and $R_N^{(\operatorname{Gi})} (z), R_N^{(\operatorname{Gi}')} (z) = \mathcal{O}_N (|z|^{ - 3N} )$ as $z\to \infty$ in the sector $|\arg z|\leq \frac{\pi}{3}-\delta$, with $\delta>0$ being fixed (see, e.g., \cite[\S9.12(viii)]{NIST}). For large positive values of $z$, the asymptotic expansion \eqref{eq73} was derived in 1950 by Scorer \cite{Scorer}. Extension to complex $z$ and the corresponding expansion \eqref{eq74} for the derivative is due to Luke \cite[p. 138]{Luke} from 1962. The second Scorer function is expressible in terms of the first function \cite[eq. 9.12.12]{NIST}:
\[
\operatorname{Gi}(z) = \frac{1}{2}\left( e^{ - \frac{\pi}{3}i} \operatorname{Hi}(ze^{\frac{2\pi}{3}i} ) + e^{\frac{\pi}{3}i} \operatorname{Hi}(ze^{ - \frac{2\pi}{3}i} ) \right).
\]
By differentiating both sides of this equality, we also have
\[
\operatorname{Gi}'(z) = \frac{1}{2}\left( e^{\frac{\pi}{3}i} \operatorname{Hi}'(ze^{\frac{2\pi}{3}i} ) + e^{ - \frac{\pi }{3}i} \operatorname{Hi}'(ze^{ - \frac{2\pi}{3}i} ) \right).
\]
If we now substitute the expansions \eqref{eq71} and \eqref{eq72} into these relations and use \eqref{eq77}, we obtain that
\[
R_N^{(\operatorname{Gi})} (z) = \frac{1}{2}\left( e^{ - \frac{\pi}{3}i} R_N^{(S)} ( \tfrac{2}{3}z^{\frac{3}{2}} e^{ - \frac{\pi}{2}i} ,0,\tfrac{1}{3}) + e^{\frac{\pi}{3}i} R_N^{(S)} ( \tfrac{2}{3}z^{\frac{3}{2}} e^{\frac{\pi}{2}i} ,0,\tfrac{1}{3}) \right)
\]
and
\[
R_N^{(\operatorname{Gi}')} (z) = \frac{1}{2}\left( {e^{\frac{\pi}{3}i} R_N^{(S)} (\tfrac{2}{3}z^{\frac{3}{2}} e^{ - \frac{\pi}{2}i} , - 1,\tfrac{2}{3}) + e^{ - \frac{\pi}{3}i} R_N^{(S)} (\tfrac{2}{3}z^{\frac{3}{2}} e^{\frac{\pi}{2}i} , - 1,\tfrac{2}{3})} \right),
\]
respectively.

The asymptotic expansions of the Struve function ${\bf H}_\nu  (z)$ and the modified Struve function ${\bf L}_\nu  (z)$ can be written
\begin{equation}\label{eq81}
{\bf H}_\nu  (z) =  Y_\nu  (z) + \frac{1}{\pi }\left( \frac{1}{2}z \right)^{\nu  - 1} \left( \sum\limits_{n = 0}^{N - 1} \frac{\Gamma \left( n + \frac{1}{2} \right)}{\Gamma \left( \nu  + \frac{1}{2} - n \right)\left( \frac{1}{2}z \right)^{2n} }  + R_N^{({\bf H})} (z,\nu ) \right)
\end{equation}
and
\begin{equation}\label{eq80}
{\bf L}_\nu  (z) = I_\nu  (z) \pm \frac{2}{\pi i}e^{ \pm \pi i\nu } K_\nu  (z) + \frac{1}{\pi}\left( \frac{1}{2}z \right)^{\nu  - 1} \left( \sum\limits_{n = 0}^{N - 1} \frac{( - 1)^{n + 1} \Gamma \left( n + \frac{1}{2} \right)}{\Gamma \left( \nu  + \frac{1}{2} - n \right)\left( \frac{1}{2}z \right)^{2n} }  + R_{N, \pm }^{({\bf L})} (z,\nu ) \right),
\end{equation}
for any $N\geq 0$, where $R_N^{({\bf H})} (z,\nu ) = \mathcal{O}_{\nu ,N} (|z|^{ - 2N} )$ and $R_{N, \pm }^{({\bf L})} (z,\nu ) = \mathcal{O}_{\nu ,N} (|z|^{ - 2N} )$ as $z\to \infty$ in the sectors $|\arg z|\leq \pi-\delta$ and $-\frac{\pi}{2} +\delta \leq \pm \arg z\leq \frac{3\pi}{2} -\delta$, with any fixed $\delta>0$ \cite[\S11.6(i)]{NIST}. The asymptotic expansion of ${\bf H}_\nu  (z)$ was given in 1887 by Rayleigh \cite{Rayleigh} for the case $\nu = 0$ and in 1882 by Struve \cite{Struve} for the case $\nu = 1$. The result for arbitrary values of $\nu$ was proved by Walker \cite[pp. 394--395]{Walker} in 1904 (see also \cite[\S 10.42]{Watson}). The asymptotic expansion of ${\bf L}_\nu  (z)$ is usually stated without the term involving $K_\nu (z)$, which is permitted if we restrict $z$ to the smaller sector $|\arg z|\leq \frac{\pi}{2}-\delta$. The Struve function is related to the Lommel function through the functional equation \cite[eq. (5), p. 80]{Luke}
\begin{equation}\label{eq84}
{\bf H}_\nu  (z) = Y_\nu  (z) + \frac{S_{\nu ,\nu } (z)}{2^{\nu  - 1} \sqrt \pi  \Gamma \left( \nu  + \frac{1}{2} \right)},
\end{equation}
which, after using \eqref{eq22}, gives
\begin{equation}\label{eq78}
R_N^{({\bf H})} (z,\nu ) = \frac{\sqrt \pi}{\Gamma \left( \nu  + \frac{1}{2} \right)}R_N^{(S)} (z,\nu ,\nu ) .
\end{equation}
A combination of \eqref{eq78} and \eqref{eq79}, for example, reproduces Watson's error estimate for the asymptotic expansion of ${\bf H}_\nu  (z)$ \cite[p. 333]{Watson}.

To derive the corresponding representations for the remainder terms $R_{N, \pm }^{({\bf L})} (z,\nu )$ in \eqref{eq80}, we combine the functional equation \cite[eqs. 11.2.2, 11.2.5 and 11.2.6]{NIST}
\begin{equation}\label{eq85}
{\bf L}_\nu  (z) = I_\nu  (z) \pm \frac{2}{\pi i}e^{ \pm \pi i\nu } K_\nu  (z) \pm ie^{ \pm \frac{\pi }{2}i\nu } ({\bf H}_\nu  (ze^{ \mp \frac{\pi }{2}i} ) - Y_\nu  (ze^{ \mp \frac{\pi }{2}i} ))
\end{equation}
with the equalities \eqref{eq81} and \eqref{eq78}. In this way, we obtain
\[
R_{N, \pm }^{({\bf L})} (z,\nu ) =  - R_N^{({\bf H})} (ze^{ \mp \frac{\pi }{2}i} ,\nu ) =  - \frac{\sqrt \pi}{\Gamma \left( \nu  + \frac{1}{2} \right)}R_N^{(S)} (ze^{ \mp \frac{\pi }{2}i} ,\nu ,\nu ).
\]

By taking the derivative of both sides of the equalities \eqref{eq81} and \eqref{eq80} with respect to $z$, it is seen that the functions ${\bf H}'_\nu  (z)$ and ${\bf L}'_\nu  (z)$ admit expansions of the form
\begin{equation}\label{eq82}
{\bf H}'_\nu  (z) =  Y'_\nu  (z) + \frac{1}{\pi }\left( \frac{1}{2}z \right)^{\nu  - 2} \left( \sum\limits_{n = 0}^{N - 1} \frac{\Gamma \left( n + \frac{1}{2} \right)\left( \frac{\nu }{2} - \frac{1}{2} - n \right)}{\Gamma \left( \nu  + \frac{1}{2} - n \right)\left( \frac{1}{2}z \right)^{2n} }  + R_N^{({\bf H}')} (z,\nu ) \right)
\end{equation}
and
\begin{equation}\label{eq83}
{\bf L}'_\nu  (z) = I'_\nu  (z) \pm \frac{2}{\pi i}e^{ \pm \pi i\nu } K'_\nu  (z) + \frac{1}{\pi}\left( \frac{1}{2}z \right)^{\nu  - 2} \left( \sum\limits_{n = 0}^{N - 1} \frac{( - 1)^{n + 1} \Gamma \left( n + \frac{1}{2} \right)\left( \frac{\nu }{2} - \frac{1}{2} - n \right)}{\Gamma \left( \nu  + \frac{1}{2} - n \right)\left( \frac{1}{2}z \right)^{2n} }  + R_{N, \pm }^{({\bf L}')} (z,\nu ) \right).
\end{equation}
If we differentiate each side of \eqref{eq84} with respect to $z$, substitute \eqref{eq76} into the right-hand side of the resulting equality and compare it with \eqref{eq82}, we find that
\begin{equation}\label{eq97}
R_N^{({\bf H'})} (z,\nu ) = \frac{\sqrt \pi  }{2\Gamma \left( \nu  + \frac{1}{2} \right)}R_N^{(S')} (z,\nu ,\nu ).
\end{equation}
To derive the analogous representations for the remainder terms $R_{N, \pm }^{({\bf L}')} (z,\nu )$ in \eqref{eq83}, we differentiate the functional equation \eqref{eq85} and combine it with the equalities \eqref{eq82} and \eqref{eq97}. In this way, we deduce
\[
R_{N, \pm }^{({\bf L'})} (z,\nu ) =  - R_N^{({\bf H'})} (ze^{ \mp \frac{\pi }{2}i} ,\nu ) =  - \frac{\sqrt \pi  }{2\Gamma \left( \nu  + \frac{1}{2} \right)}R_N^{(S')} (ze^{ \mp \frac{\pi }{2}i} ,\nu ,\nu ).
\]
In particular, $R_N^{({\bf H}')} (z,\nu ) = \mathcal{O}_{\nu ,N} (|z|^{ - 2N} )$ and $R_{N, \pm }^{({\bf L}')} (z,\nu ) = \mathcal{O}_{\nu ,N} (|z|^{ - 2N} )$ as $z\to \infty$ in the sectors $|\arg z|\leq \pi-\delta$ and $-\frac{\pi}{2} +\delta \leq \pm \arg z\leq \frac{3\pi}{2} -\delta$, with any fixed $\delta>0$. Consequently, the expansions \eqref{eq82} and \eqref{eq83} are the large-$z$ asymptotic expansions of the derivatives of the functions ${\bf H}_\nu  (z)$ and ${\bf L}_\nu  (z)$.

\section{Discussion}\label{section6}

In this paper, we have derived new integral representations and estimates for the remainder terms of the large-argument asymptotic expansions of the Lommel function and its derivative. We have also constructed error bounds for the re-expansions of these remainders. In this section, we shall discuss the sharpness of our error bounds. For the sake of brevity, we consider only the bounds for $R_N^{(S)} (z,\mu ,\nu )$ and $R_{N,M}^{(S)} (z,\mu ,\nu )$; the other remainder terms can be treated in a similar manner.

First, we consider the bounds for $R_N^{(S)} (z,\mu ,\nu )$ with $\mu$ and $\nu$ being real. Let $N$ be any non-negative integer, and let $\mu$ and $\nu$ be arbitrary real numbers such that $\mu + |\nu| < 2N + 1$. Under these assumptions, it follows from Theorem \ref{thm4} (or Theorem \ref{thm6}) and Propositions \ref{prop1}, \ref{prop2} and \ref{prop3} that
\begin{equation}\label{eq90}
\big| R_N^{(S)} (z,\mu ,\nu ) \big| \le \frac{\left| a_N ( - \mu ,\nu ) \right|}{\left| z \right|^{2N}} \times \begin{cases} 1 & \text{ if } \; \left|\arg z\right| \leq \frac{\pi}{4}, \\ \min\Big(\left|\csc ( 2\arg z)\right|,1 + \cfrac{1}{2}\chi(2N -\mu+ 1)\Big) & \text{ if } \; \frac{\pi}{4} < \left|\arg z\right| \leq \frac{\pi}{2}, \\ \cfrac{\sqrt {2\pi (2N -\mu+ 1)} }{2\left| \sin (\arg z) \right|^{2N -\mu+ 1} } + 1 + \cfrac{1}{2}\chi (2N -\mu+ 1) & \text{ if } \; \frac{\pi}{2} < \left|\arg z\right| < \pi. \end{cases}
\end{equation}
Here, following Olver \cite{Olver}, we use the notation
\begin{equation}\label{eq96}
\chi (p) = \pi ^{\frac{1}{2}} \frac{\Gamma \left( \frac{p}{2} + 1 \right)}{\Gamma \left( \frac{p}{2} + \frac{1}{2} \right)},
\end{equation}
for any $p >0$.

By the definition of an asymptotic expansion, $\lim _{z \to \infty } \left| z \right|^{2N} \big| R_N^{(S)} (z,\mu ,\nu ) \big| = \left| a_N ( - \mu ,\nu ) \right|$ for any fixed $N \geq 0$. Therefore, when $\left|\arg z\right| \leq \frac{\pi}{4}$, the estimate \eqref{eq90} and hence our error bound \eqref{eq60} (or \eqref{eq43}) cannot be improved, in general.

Consider now the case when $\frac{\pi}{4} < \left|\arg z\right| \leq \frac{\pi}{2}$. The bound \eqref{eq90} is reasonably sharp as long as $| \csc(2 \arg z)|$ is not very large, i.e., when $| \arg z|$ is bounded away from $\frac{\pi}{2}$. As $|\arg z|$ approaches $\frac{\pi}{2}$, the factor $|\csc(2 \arg z)|$ grows indefinitely, and therefore, it has to be replaced by $1 + \frac{1}{2}\chi(2N -\mu+ 1)$. By Stirling's formula,
\begin{equation}\label{eq91}
\chi(2N -\mu+ 1) = \sqrt {\frac{\pi (2N - \mu  + 1)}{2}} \left( 1 + \mathcal{O}_\mu  \left( \frac{1}{N} \right)\right)
\end{equation}
as $N\to +\infty$, and therefore the appearance of this factor in the bound \eqref{eq90} may give the impression that this estimate is unrealistic for large $N$. However, this is not the case, as the following argument shows. We assume that $\mu$ and $\nu$ are fixed, $|2N - \mu  + 1 - |z||$ is bounded, $| \arg z| = \frac{\pi}{2}$ and $|z|$ is large. Under these assumptions, Theorem \ref{thm7}, Proposition \ref{prop2} and Stirling's formula imply that
\begin{align*}
& \big| R_N^{(S)} (z,\mu ,\nu )\big| = \frac{2^{\mu  + \frac{1}{2}} \pi ^{\frac{1}{2}} \Gamma \left( 2N - \mu  + \frac{1}{2} \right)}{\left| \Gamma \left( {\frac{ - \mu  + \nu  + 1}{2}} \right)\Gamma \left( \frac{ - \mu  - \nu  + 1}{2} \right) \right|}\frac{1}{\left| z \right|^{2N}}\left( \Pi _{2N - \mu  + \frac{1}{2}} (z) + \mathcal{O}_{\mu ,\nu } \left( \frac{1}{\left| z \right|^{\frac{1}{2}}} \right) \right)
\\ &  = \frac{\left| a_N ( - \mu ,\nu ) \right|}{\left| z \right|^{2N}}\frac{2^{\mu  + \frac{1}{2}} \pi ^{\frac{1}{2}} \Gamma \left( 2N - \mu  + \frac{1}{2} \right)}{2^{2N} \Gamma \left( \frac{ - \mu  + \nu  + 1}{2} + N \right)\Gamma \left( \frac{ - \mu  - \nu  + 1}{2} + N \right)} \left( \Pi _{2N - \mu  + \frac{1}{2}} (z) + \mathcal{O}_{\mu ,\nu } \left( \frac{1}{\left| z \right|^{\frac{1}{2}} } \right) \right)
\\ & = \frac{\left| a_N ( - \mu ,\nu ) \right|}{\left| z \right|^{2N}}\left( 1 + \mathcal{O}_{\mu ,\nu } \left( \frac{1}{\left| z \right|} \right) \right)\left( \Pi _{2N - \mu  + \frac{1}{2}} (z) + \mathcal{O}_{\mu ,\nu } \left( \frac{1}{\left| z \right|^{\frac{1}{2}} } \right) \right).
\end{align*}
It is known that when $| \arg w| = \frac{\pi}{2}$ and $|p - |w||$ is bounded, the asymptotic approximation
\[
\Pi _p (w) = \frac{1}{2}\sqrt {\frac{\pi}{2}\left( p + \frac{1}{2} \right)} \left( 1 + \mathcal{O}\left( \frac{1}{\left| w \right|^{\frac{1}{2}} } \right) \right)
\]
holds as $|w| \to +\infty$ \cite{Nemes3}. Consequently,
\begin{equation}\label{eq92}
\big| R_N^{(S)} (z,\mu ,\nu ) \big| = \frac{\left| a_N ( - \mu ,\nu ) \right|}{\left| z \right|^{2N}}\frac{1}{2}\sqrt {\frac{\pi (2N - \mu  + 1)}{2}} \left( 1 + \mathcal{O}_{\mu ,\nu } \left( \frac{1}{\left| z \right|^{\frac{1}{2}}} \right) \right).
\end{equation}
On the other hand, the bound \eqref{eq90} and the asymptotic formula \eqref{eq91} show that
\begin{align*}
\big| R_N^{(S)} (z,\mu ,\nu ) \big| & \le \frac{\left| a_N ( - \mu ,\nu ) \right|}{\left| z \right|^{2N} }\left( 1 + \frac{1}{2}\sqrt {\frac{\pi (2N - \mu  + 1)}{2}} \left( 1 + \mathcal{O}_\mu  \left( \frac{1}{N} \right)\right) \right)
\\ & = \frac{\left| a_N ( - \mu ,\nu ) \right|}{\left| z \right|^{2N} }\frac{1}{2}\sqrt {\frac{\pi (2N - \mu  + 1)}{2}} \left( 1 + \mathcal{O}_\mu  \left( \frac{1}{\left| z \right|^{\frac{1}{2}} } \right) \right).
\end{align*}
Therefore, when $|\arg z| = \frac{\pi}{2}$, the upper bound in \eqref{eq90} is asymptotically equal to the right-hand side of the equality \eqref{eq92}. Consequently, when $| \arg z|$ is equal or smaller but close to $\frac{\pi}{2}$, the estimate \eqref{eq90} and thus the error bound \eqref{eq60} (or \eqref{eq43}) cannot be improved in general.

Finally, assume that $\frac{\pi}{2} < \left|\arg z\right| < \pi$. Elementary analysis shows that the factor
\[
\frac{\sqrt {2\pi (2N -\mu+ 1)} }{2\left| \sin (\arg z)\right|^{2N -\mu+ 1} } + 1 + \frac{1}{2}\chi (2N -\mu+ 1),
\]
as a function of $N$, remains bounded, provided that $|\arg z| - \frac{\pi}{2}=\mathcal{O}(N^{-\frac{1}{2}})$. This gives a reasonable estimate for the remainder term $R_N^{(S)} (z,\mu ,\nu )$. Otherwise, $\left| \sin (\arg z)\right|^{2N -\mu+ 1}$ can take very small values when $N$ is large, making the bound \eqref{eq90} completely unrealistic in most of the sectors $\frac{\pi}{2} < \left|\arg z\right| < \pi$. This deficiency of the bound \eqref{eq90} (and \eqref{eq60} or \eqref{eq43}) is necessary and is due to the omission of certain exponentially small terms arising from the Stokes phenomenon related to the asymptotic expansion \eqref{eq1} of the Lommel function. Thus, the use of the asymptotic expansion \eqref{eq1} should be confined to the sector $|\arg z|\leq \frac{\pi}{2}$. For other ranges of $\arg z$, one should use analytic continuation formulae for the Lommel function.

We continue by discussing the bound \eqref{eq43} which covers the case of complex $\mu$ and $\nu$. It is natural to compare the right-hand side of \eqref{eq43} to the first neglected term. Therefore, we begin by expressing the quantity $a_N ( - \Re (\mu ),\Re (\nu ))$ in terms of $a_N ( - \mu ,\nu )$:
\begin{gather}\label{eq101}
\begin{split}
a_N ( - \Re (\mu ),\Re (\nu )) = \; & \frac{\Gamma \big( \frac{ - \Re (\mu ) + \Re (\nu ) + 1}{2} + N \big)\Gamma \big( \frac{ - \Re (\mu ) - \Re (\nu ) + 1}{2} + N \big)}{\Gamma \left( \frac{ - \mu  + \nu  + 1}{2} + N \right)\Gamma \left( \frac{ - \mu  - \nu  + 1}{2} + N \right)} \\ & \times \frac{\Gamma \left( \frac{ - \mu  + \nu  + 1}{2} \right)\Gamma \left( \frac{ - \mu  - \nu  + 1}{2} \right)}{\Gamma \big( \frac{ - \Re (\mu ) + \Re (\nu ) + 1}{2} \big)\Gamma \big( \frac{ - \Re (\mu ) - \Re (\nu ) + 1}{2} \big)}a_N ( - \mu ,\nu ).
\end{split}
\end{gather}
Consider first the case when $\left|\arg z\right| < \frac{\pi}{2}$. Let $N$ be a non-negative integer and let $\mu$ and $\nu$ be complex numbers such that $\Re(\mu)+|\Re(\nu)|<2N+1$. Under these assumptions, it follows from Theorem \ref{thm6}, Proposition \ref{prop1} and \eqref{eq101} that
\begin{gather}\label{eq100}
\begin{split}
\big| R_N^{(S)} (z,\mu ,\nu ) \big| \le \; & \left| \frac{\Gamma \big( \frac{ - \Re (\mu ) + \Re (\nu ) + 1}{2} + N \big)\Gamma \big( \frac{ - \Re (\mu ) - \Re (\nu ) + 1}{2} + N \big)}{\Gamma \left( \frac{ - \mu  + \nu  + 1}{2} + N \right)\Gamma \left( \frac{ - \mu  - \nu  + 1}{2} + N \right)} \right|\frac{\left| a_N ( - \mu ,\nu ) \right|}{\left| z \right|^{2N} } \\ & \times \begin{cases} 1 & \text{ if } \; \left|\arg z\right| \leq \frac{\pi}{4}, \\ \left|\csc ( 2\arg z)\right| & \text{ if } \; \frac{\pi}{4} < \left|\arg z\right| < \frac{\pi}{2}. \end{cases}
\end{split}
\end{gather}
Now, we make the assumptions that $2N - \Re (\mu ) \pm \Re (\nu ) \to  + \infty$ and $\Im (\mu ) \pm \Im (\nu ) = o(N^{\frac{1}{2}} )$ as $N\to +\infty$. With these provisos, it can easily be shown, using for example Stirling's formula, that the quotient of gamma functions in \eqref{eq100} is asymptotically $1$ for large $N$. Consequently, if $N$ is large, the estimates \eqref{eq100} and \eqref{eq43} are reasonably sharp in most of the sector $\left|\arg z\right| < \frac{\pi}{2}$. If the assumption $\Im (\mu ) \pm \Im (\nu ) = o(N^{\frac{1}{2}} )$ is replaced by the weaker condition $\Im (\mu ) \pm \Im (\nu ) = \mathcal{O}(N^{\frac{1}{2}} )$, the quotient in (5.5) is still bounded and hence, the estimates \eqref{eq100} and \eqref{eq43} are still relatively sharp. Otherwise, if $2N - \Re (\mu ) \pm \Re (\nu )$ is small and $\Im (\mu ) \pm \Im (\nu )$ is much larger than $N^{\frac{1}{2}}$, this quotient may grow exponentially fast in $|\Im (\mu ) \pm \Im (\nu )|$, which can make the bound \eqref{eq100} completely unrealistic. Therefore, if the asymptotic expansion \eqref{eq1} is truncated just before its numerically least term, i.e., when $N\approx \frac{1}{2}|z|$, the estimate \eqref{eq100} is reasonable in most of the sector $\left|\arg z\right| < \frac{\pi}{2}$ as long as $\mu\pm\nu=\mathcal{O}(|z|^{\frac{1}{2}} )$. Although \eqref{eq1} is valid in the sense of generalized asymptotic expansions \cite[\S2.1(v)]{NIST} provided $\mu\pm\nu=o(|z|)$, the stronger assumption $\mu\pm\nu=\mathcal{O}(|z|^{\frac{1}{2}} )$ guarantees the rapid decay of the first several terms of \eqref{eq1} for large $z$.

When $|\arg z|$ is equal or smaller but close to $\frac{\pi}{2}$, the estimate \eqref{eq100} may be replaced by
\begin{gather}\label{eq102}
\begin{split}
& \big| R_N^{(S)} (z,\mu ,\nu ) \big| \le \left| \frac{\Gamma \big( \frac{ - \Re (\mu ) + \Re (\nu ) + 1}{2} + N \big)\Gamma \big( \frac{ - \Re (\mu ) - \Re (\nu ) + 1}{2} + N \big)}{\Gamma \left( \frac{ - \mu  + \nu  + 1}{2} + N \right)\Gamma \left( \frac{ - \mu  - \nu  + 1}{2} + N \right)}\frac{\Gamma (2N - \mu  + 1)}{\Gamma (2N - \Re (\mu ) + 1)} \right|\frac{\left| a_N ( - \mu ,\nu ) \right|}{\left| z \right|^{2N}}
\\ & \times \left( \frac{1}{2} + \frac{\left| 2N - \mu  + 1 \right|}{2(2N - \Re (\mu ) + 1)}\chi (2N - \Re (\mu ) + 1)\max (1,e^{ \Im (\mu )\arg z} ) + \frac{\Gamma (2N - \Re (\mu ) + 1)}{2\left| \Gamma (2N - \mu  + 1) \right|} \right),
\end{split}
\end{gather}
which can be derived from Theorem \ref{thm6}, Proposition \ref{prop2} and \eqref{eq101}. If $\mu$ and $\nu$ satisfy the requirements posed in the previous paragraph, then the ratio of gamma functions in the first line of \eqref{eq102} is asymptotically $1$ for large $N$. Also, if in addition $\Im (\mu )\arg z$ is non-positive, then the factor in the second line is $\mathcal{O}(N^{\frac{1}{2}})$ when $N$ is large. By an argument similar to that used in the case of real $\mu$ an $\nu$, it follows that \eqref{eq102} is a realistic bound. Otherwise, if $2N - \Re (\mu ) \pm \Re (\nu )$ is small and $\Im (\mu ) \pm \Im (\nu )$ is much larger than $N^{\frac{1}{2}}$ or $\Im (\mu )\arg z$ is positive and large, the right-hand side of the inequality \eqref{eq102} may be a serious overestimate of the actual error.

Similarly to the case of real $\mu$ and $\nu$, for values of $z$ outside the closed sector $\left|\arg z\right| \leq \frac{\pi}{2}$, one should make use of analytic continuation formulae.

Tables \ref{table1} and \ref{table2} below illustrate the numerical performance of the bound \eqref{eq90} for some specific values of the various parameters. It is clearly seen from Table \ref{table2} that the asymptotic expansion \eqref{eq1} becomes less efficient once $\mu \pm \nu$ is comparable to $|z|$. Similarly, Table \ref{table3} below illustrates the numerical performance of the bounds \eqref{eq100} and \eqref{eq102} for some specific values of the various parameters. The tables demonstrate well that the bounds behave in a manner that is rather similar to the actual values of the modulus of the remainder term $R_N^{(S)}(z,\mu,\nu)$.

We conclude this section with a brief discussion of the estimate \eqref{eq55} for $R_{N,M}^{(S)} (z,\mu,\nu)$. When $z$ is large, $|z|^M \big| R_M^{(K)} ( |z |,\nu) \big| \lesssim | a_M (\nu) |$ holds, and therefore the first term on the right-hand side of the inequality \eqref{eq55} is of the same order of magnitude as the first neglected term in the expansion \eqref{eq23}. It can be shown that, when $2N-M-\mu+\frac{1}{2}$ is large, the second term is comparable with, or less than, the first term (except near the zeros of $\Pi_{2N-M-\mu+\frac{1}{2}}(z)$). The proof of this fact is identical to the proof given by Boyd \cite{Boyd} for the analogous re-expansion of the remainder term $R_{N}^{(K)} (z,\nu)$ and it is therefore not pursued here. In summary, the bound \eqref{eq55} is comparable with the first neglected term in the expansion \eqref{eq23}, unless $z$ is close to a zero of $\Pi_{2N-M-\mu+\frac{1}{2}}(z)$, and thus, this bound is reasonably sharp.

\begin{table}[ht]
\begin{center}
{\renewcommand{\arraystretch}{1.4}
\begin{tabular}{ c | c|c|c|c  }
  \hline			
  $\arg z$ & $|R_5^{(S)}(z,\mu,\nu)|$ & bound \eqref{eq90} & $|R_{10}^{(S)}(z,\mu,\nu)|$ & bound \eqref{eq90} \\ \hline
  $0$ & $0.47440\times 10^{-5}$ & $0.65562 \times 10^{-5}$ & $0.48851 \times 10^{-6}$ & $1.07336 \times 10^{-6}$\\
  $\frac{\pi}{4}$ & $0.59063 \times 10^{-5}$ & $0.65562 \times 10^{-5}$ & $0.66414 \times 10^{-6}$ & $1.07336 \times 10^{-6}$\\
  $\frac{3\pi}{8}$ & $0.79416 \times 10^{-5}$ & $0.92718 \times 10^{-5}$ & $0.10712 \times 10^{-5}$ & $0.15179 \times 10^{-5}$\\
  $\frac{\pi}{2}$ & $0.12556 \times 10^{-4}$ & $0.21657 \times 10^{-4}$ & $0.28159 \times 10^{-5}$ & $0.43344 \times 10^{-5}$\\
\end{tabular}}
\end{center}
\caption{Numerical values of the modulus of the remainder term $R_N^{(S)}(z,\mu,\nu)$, with $N=5$, $10$, $|z|=20$, $\mu=-2$, $\nu=\frac{3}{2}$ and various values of $\arg z$, together with the corresponding estimates using the bound \eqref{eq90}.}
\label{table1}
\end{table}

\begin{table}[ht]
\begin{center}
{\renewcommand{\arraystretch}{1.4}
\begin{tabular}{ c | c|c|c|c  }
  \hline			
  $\arg z$ & $|R_5^{(S)}(z,\mu,\nu)|$ & bound \eqref{eq90} & $|R_{10}^{(S)}(z,\mu,\nu)|$ & bound \eqref{eq90} \\ \hline
  $0$ & $0.32502\times 10^{-3}$ & $0.52337 \times 10^{-3}$ & $0.23193 \times 10^{-3}$ & $0.60589 \times 10^{-3}$\\
  $\frac{\pi}{4}$ & $0.42580 \times 10^{-3}$ & $0.52337 \times 10^{-3}$ & $0.31254 \times 10^{-3}$ & $0.60589 \times 10^{-3}$\\
  $\frac{3\pi}{8}$ & $0.63393 \times 10^{-3}$ & $0.74016 \times 10^{-3}$ & $0.49243 \times 10^{-3}$ & $0.85687 \times 10^{-3}$\\
  $\frac{\pi}{2}$ & $0.13345 \times 10^{-2}$ & $0.18957 \times 10^{-2}$ & $0.11804 \times 10^{-2}$ & $0.25972 \times 10^{-2}$\\
\end{tabular}}
\end{center}
\caption{Numerical values of the modulus of the remainder term $R_N^{(S)}(z,\mu,\nu)$, with $N=5$, $10$, $|z|=20$, $\mu=-6$, $\nu=\frac{9}{2}$ and various values of $\arg z$, together with the corresponding estimates using the bound \eqref{eq90}.}
\label{table2}
\end{table}

\begin{table}[ht]
\begin{center}
{\renewcommand{\arraystretch}{1.4}
\begin{tabular}{ c | c|c|c|c  }
  \hline			
  $\arg z$ & $|R_5^{(S)}(z,\mu,\nu)|$ & bound min(\eqref{eq100},\eqref{eq102}) & $|R_{10}^{(S)}(z,\mu,\nu)|$ & bound min(\eqref{eq100},\eqref{eq102}) \\ \hline
  $0$ & $0.32299\times 10^{-7}$ & $0.51174 \times 10^{-7}$ & $0.25597 \times 10^{-9}$ & $0.53363 \times 10^{-9}$\\
  $\frac{\pi}{4}$ & $0.40084 \times 10^{-7}$ & $0.51174 \times 10^{-7}$ & $0.37853 \times 10^{-9}$ & $0.53363 \times 10^{-9}$\\
  $\frac{3\pi}{8}$ & $0.47580 \times 10^{-7}$ & $0.72371 \times 10^{-7}$ & $0.65102 \times 10^{-9}$ & $0.75467 \times 10^{-9}$\\
  $\frac{\pi}{2}$ & $0.50481 \times 10^{-7}$ & $19.03354 \times 10^{-7}\,\;$ & $0.19565 \times 10^{-8}$ & $3.13582 \times 10^{-8}$\\
\end{tabular}}
\end{center}
\caption{Numerical values of the modulus of the remainder term $R_N^{(S)}(z,\mu,\nu)$, with $N=5$, $10$, $|z|=20$, $\mu=2+2i$, $\nu=\frac{1}{2}-i$ and various values of $\arg z$, together with the corresponding estimates using the minimum of the bounds \eqref{eq100} and \eqref{eq102}.}
\label{table3}
\end{table}

\section*{Acknowledgement} The author's research was supported by a research grant (GRANT11863412/70NANB15H221) from the National Institute of Standards and Technology. The author thanks the anonymous referees for their helpful comments and suggestions on the manuscript.

\appendix

\section{Bounds for the basic terminant}\label{appendixa}

In this appendix, we give some estimates for the absolute value of the basic terminant $\Pi _p (w)$ with $\Re(p) > 0$. These estimates depend only on $p$ and the argument of $w$ and therefore also provide bounds for the quantity $\mathop {\sup }\nolimits_{r \ge 1} \left| {\Pi _{p} ( zr)} \right|$ which appears in Theorems \ref{thm4} and \ref{thm6}. For proofs, see \cite{Nemes2} and \cite{Nemes3}.

\begin{proposition}\label{prop1}
For any complex $p$ with $\Re(p) > 0$ it holds that
\[
\left| {\Pi _p (w)} \right| \le \frac{\Gamma (\Re (p))}{\left| \Gamma (p) \right|} \times \begin{cases} 1 & \text{ if } \; \left|\arg w\right| \leq \frac{\pi}{4}, \\ \left|\csc ( 2\arg w)\right| & \text{ if } \; \frac{\pi}{4} < \left|\arg w\right| < \frac{\pi}{2}. \end{cases}
\]
Moreover, when $w$ and $p$ are positive, we have $0<\Pi _p (w)<1$.
\end{proposition}

We remark that it was shown by the author \cite{Nemes2} that $\left| {\Pi _p (w)} \right| \le \sqrt {\frac{e}{4}\left( {p + \frac{3}{2}} \right)}$ provided that $\frac{\pi }{4} < \left| {\arg w} \right| \le \frac{\pi }{2}$ and $p$ is real and positive, which improves on the bound \eqref{eq65} near $|\arg w|=\frac{\pi}{2}$.

\begin{proposition} For any complex $p$ with $\Re(p) > 0$, we have
\[
\left| \Pi_p (w) \right| \le \frac{1}{2}\sec ^{\Re(p)} (\arg w)\max (1,e^{\Im (p)\left(  \mp \frac{\pi }{2} - \arg w \right)} ) + \frac{1}{2}\max (1,e^{\Im (p)\left(  \pm \frac{\pi }{2} - \arg w \right)} )
\]
and
\[
\left| \Pi_p (w) \right| \le \frac{1}{2}\sec ^{\Re(p)} (\arg w)\max (1,e^{\Im (p)\left(  \mp \frac{\pi }{2} - \arg w \right)} ) + \frac{\Gamma (\Re (p))}{2\left| \Gamma (p) \right|},
\]
for $0 \le \pm\arg w < \frac{\pi }{2}$.
\end{proposition}

The following estimate is valid for positive real values of the order $p$.

\begin{proposition} For any $p>0$ and $w$ with $\frac{\pi}{4} < \left| {\arg w} \right| < \pi$, we have
\begin{equation}\label{eq86}
\left| \Pi _p \left( w \right) \right| \le \frac{\left| \csc \left( 2\left( \arg w - \varphi \right) \right) \right|}{\cos ^p \varphi },
\end{equation}
where $\varphi$ is the unique solution of the implicit equation
\[
\left( {p + 2} \right)\cos \left( 2\arg w -3\varphi \right) = \left( p - 2 \right)\cos\left(2\arg w-\varphi\right)
\]
that satisfies $0 < \varphi  <  - \frac{\pi}{4} + \arg w$ if $\frac{\pi}{4} < \arg w  < \frac{\pi}{2}$, $ - \frac{\pi}{2}  + \arg w  < \varphi  < -\frac{\pi}{4}+\arg w$ if $\frac{\pi}{2}  \le \arg w  < \frac{3\pi}{4}$, $ - \frac{\pi}{2}  + \arg w  < \varphi  < \frac{\pi }{2}$ if $\frac{3\pi}{4}  \le \arg w  < \pi$, $\frac{\pi}{4} + \arg w < \varphi  <  0$ if $-\frac{\pi }{2} < \arg w  < -\frac{\pi}{4}$, $\frac{\pi}{4}  + \arg w  < \varphi  < \frac{\pi}{2}+\arg w$ if $-\frac{3\pi}{4}  < \arg w  \le -\frac{\pi}{2}$ and $ - \frac{\pi}{2}  < \varphi  < \frac{\pi }{2}+ \arg w$ if $-\pi < \arg w \le -\frac{3\pi}{4}$.
\end{proposition}

We remark that the value of $\varphi$ in this proposition is chosen so as to minimize the right-hand side of the inequality \eqref{eq86}.

\begin{proposition}\label{prop2} For any complex $p$ with $\Re(p)>0$, we have
\begin{gather}\label{eq87}
\begin{split}
\left| \Pi _p (w) \right| \le \; & \frac{1}{2}  + \frac{\left| p \right|}{2 {\Re (p)} }\Gamma \left( \frac{\Re (p)}{2} + 1 \right) {\bf F} \left( \frac{1}{2},\frac{\Re (p)}{2};\frac{\Re (p)}{2} + 1;\sin ^2 (\arg w) \right)\max (1,e^{ - \Im (p)\arg w} ) \\ & + \frac{1}{2}\max (1,e^{ \Im (p)\left( { \pm\frac{\pi }{2} - \arg w} \right)} )
\\  \le \; & \frac{1}{2} + \frac{\left| p \right|}{2\Re (p)}\chi (\Re (p))\max (1,e^{ - \Im (p)\arg w} ) + \frac{1}{2}\max (1,e^{\Im (p)\left(  \pm \frac{\pi }{2} - \arg w \right)} )
\end{split}
\end{gather}
and
\begin{align*}
\left| \Pi _p (w) \right| \le \; & \frac{1}{2}  + \frac{\left| p \right|}{2 {\Re (p)} } \Gamma \left( \frac{\Re (p)}{2} + 1 \right) {\bf F} \left( \frac{1}{2},\frac{\Re (p)}{2};\frac{\Re (p)}{2} + 1;\sin ^2 (\arg w) \right)\max (1,e^{ - \Im (p)\arg w} ) \\ & + \frac{\Gamma (\Re (p))}{2\left| \Gamma (p) \right|}
\\  \le \; & \frac{1}{2} + \frac{\left| p \right|}{2\Re (p)}\chi (\Re (p))\max (1,e^{ - \Im (p)\arg w} ) + \frac{\Gamma (\Re (p))}{2\left| \Gamma (p) \right|},
\end{align*}
for $\frac{\pi }{4} < \pm \arg w \le \frac{\pi }{2}$. Here $\chi(p)$ is defined by \eqref{eq96}.
\end{proposition}

\begin{proposition}\label{prop3} For any complex $p$ with $\Re(p)>0$, we have
\begin{gather}\label{eq105}
\begin{split}
\left| \Pi _p (w) \right| & \le e^{\Im (p)\left( { \pm \frac{\pi }{2} - \arg w} \right)} \frac{\Gamma (\Re (p))}{\left| {\Gamma (p)} \right|}\frac{{\sqrt {2\pi \Re (p)} }}{{2\left| {\sin (\arg w)} \right|^{\Re (p)} }} + |\Pi _p (we^{ \mp \pi i} )| \\ & \le e^{\Im (p)\left( { \pm \frac{\pi }{2} - \arg w} \right)} \frac{\Gamma (\Re (p))}{\left| \Gamma (p) \right|}\frac{\chi (\Re (p))}{\left| \sin (\arg w) \right|^{\Re (p)} } + |\Pi _p (we^{ \mp \pi i} )|,
\end{split}
\end{gather}
for $\frac{\pi}{2}<\pm \arg w <\pi$.
\end{proposition}

The dependence on $|w|$ in the estimates \eqref{eq105} may be eliminated by employing the bounds for $|\Pi _p (we^{ \mp \pi i} )|$ that were given previously.

Finally, we mention the following two-sided inequality proved by Watson \cite{Watson2} for positive real values of $p$:
\[
\sqrt {\frac{\pi }{2}\Big( p + \frac{1}{2} \Big)}  < \chi(p) < \sqrt {\frac{\pi }{2}\Big( p + \frac{2}{\pi} \Big)} .
\]
The upper inequality can be used to simplify the error bounds involving $\chi(p)$.

\end{document}